\documentclass[12pt]{amsproc}
\usepackage{amsmath}
\usepackage[fixamsmath]{mathtools}
\usepackage{amsbsy}
\usepackage{amssymb}
\usepackage{latexsym}
\usepackage{mathrsfs}
\usepackage[usenames]{color}
\usepackage{fullpage}
\usepackage{setspace}

\usepackage[colorlinks=true,citecolor=green,linkcolor=red,urlcolor=red]{hyperref}
\usepackage{breakurl}

\allowdisplaybreaks

\def\bF{{\mathbf F}}
\def\bbC{{\mathbb C}}
\def\C{{\mathbb C}}

\def\Z{{\mathbb Z}}
\def\bbN{{\mathbb N}}

\def\bbR{{\mathbb R}}

\def\bfone{{\boldsymbol 1}}
\def\sX{{\mathscr X}}
\def\sY{{\mathscr Y}}

\newtheorem{thm}{Theorem}[section]
\newtheorem{lem}[thm]{Lemma}
\newtheorem{dfn}[thm]{Definition}
\newtheorem{cor}[thm]{Corollary}
\newtheorem{pro}[thm]{Proposition}

\newtheorem{remark}[thm]{Remark}

\def\verts{\mathsf{V}}
\def\vertex{\mathsf{V}}
\def\edges{\mathsf{E}}
\def\edge{\mathsf{E}}
\def\Hilb{{\mathscr H}}
\def\fd{{\mathscr F}}
\def\E{{\bf E}}
\def\P{{\bf P}}

\def\dfnterm#1{\textit{\textbf{#1}}}

\def\st{\,;\;}
\def\arXiv#1{\url{http://www.arxiv.org/abs/#1}}
\def\restrict{\mathord{\upharpoonright}}   
\def\gp{{\Gamma}}
\def\gpe{{\gamma}}
\def\dom{\preccurlyeq}   
\font\frak=eufm10   
\font\scriptfrak=eufm7
\font\scriptscriptfrak=eufm5
\def\mathfrak#1#2{
\def#1{{\mathchoice%
{{\hbox{\frak #2}}}%
{{\hbox{\frak #2}}}%
{{\hbox{\scriptfrak #2}}}%
{{\hbox{\scriptscriptfrak #2}}}}}}
\mathfrak{\ba}{B}  
\mathfrak{\qba}{S}  
\mathfrak{\fo}{F}  
\def\fsf{\mathsf{FSF}}         
\def\FSF{\mathsf{FSF}}         
\def\wsf{\mathsf{WSF}}         
\def\WSF{\mathsf{WSF}}         
\def\ust{\mathsf{UST}}         
\def\STAR{\bigstar}        
\def\CYCLE{\diamondsuit}
\def\cbuldot{{\raise.25ex\hbox{$\scriptscriptstyle\bullet$}}}
\def\iprod#1{\langle #1 \rangle}  
\def\FS{\bF_S}   
\def\gh{G}  
\def\marks{{\Xi}}  
\def\mk{\xi}  
\def\bp{o}
\def\rtd{\mu}  
\def\GG{{\mathcal G}}
\def\sGG{{\mathcal G}_{\star}}
\def\cd{\Rightarrow}  
\def\gtwo{{{\mathcal G}_{**}}}  
 
\def\flr#1{\lfloor #1 \rfloor}   

\DeclareMathOperator{\tr}{tr}  
\def\im{{\rm im}}  
\def\dbar{\bar d}
\def\Pbig#1{\P\big[#1\big]}

\def\BLPSusf{\cite{BLPS:usf}}

\def\thmenv#1#2#3{\begin{#1} \label{#1:#2} #3 \end{#1}}

\def\procl#1.#2 #3\endprocl{%
       \ifx#1t\thmenv{thm}{#2}{#3}\fi
       \ifx#1l\thmenv{lem}{#2}{#3}\fi
       \ifx#1p\thmenv{pro}{#2}{#3}\fi
       \ifx#1c\thmenv{cor}{#2}{#3}\fi
       \ifx#1d\thmenv{dfn}{#2}{#3}\fi
       \ifx#1g\thmenv{conj}{#2}{#3}\fi
       \ifx#1q\thmenv{question}{#2}{#3}\fi
       \ifx#1r\thmenv{remark}{#2}{{\rm #3}}\fi
    }%

\def\rref#1.#2/{%
      \ifx #1sSection~\ref{s.#2}\fi
      \ifx #1tTheorem~\ref{thm:#2}\fi  
      \ifx #1lLemma~\ref{lem:#2}\fi 
      \ifx #1cCorollary~\ref{cor:#2}\fi 
      \ifx #1pProposition~\ref{pro:#2}\fi 
      \ifx #1dDefinition~\ref{dfn:#2}\fi
      \ifx #1gConjecture~\ref{conj:#2}\fi 
      \ifx #1qQuestion~\ref{question:#2}\fi 
      \ifx #1rRemark~\ref{remark:#2}\fi 
      \ifx #1aAppendix~\ref{a.#2}\fi 
      \ifx #1fFigure~\ref{f.#2}\fi
      \ifx #1e(\eqref{e.#2})\fi
      \ifx #1b\cite{#2}\fi
        }

\def\rlabel #1 #2{\begin{equation} \label{#1} #2 \end{equation}}

\def\rproof{\begin{proof}}

\def\Qed{\end{proof}}

\def\bsection#1#2{\bigbreak\section{#1}\label{#2}}

\begin{document}

\onehalfspacing

\title[Invariant Couplings]{Invariant Coupling of Determinantal Measures \\ on Sofic Groups}
\author{Russell Lyons}
\address{R.L., Dept.\ of Math., 831 E. 3rd St., Indiana Univ., Bloomington, IN 47405-7106, USA}
\email{rdlyons@indiana.edu}

\author{Andreas Thom}
\address{A.T., Math. Institut, Univ. Leipzig, PF 100920, D-04009 Leipzig, Germany
}
\email{thom@math.uni-leipzig.de}

\thanks{R.L.'s research is
partially supported by NSF grant DMS-1007244 and Microsoft Research. A.T.'s research is supported by ERC-StG 277728.}

\date{15 May 2014}  

\begin{abstract}
\baselineskip=12pt
To any positive contraction $Q$ on $\ell^2(W)$, there is associated a
determinantal probability measure $\P^Q$ on $2^W$, where $W$ is a
denumerable set.
Let $\gp$ be a countable sofic finitely generated
group and $G = (\gp, \edge)$ be a Cayley graph of $\gp$.
We show that if $Q_1$ and $Q_2$ are two $\gp$-equivariant positive
contractions on $\ell^2(\gp)$ or on $\ell^2(\edge)$ with $Q_1 \le Q_2$,
then there exists a $\gp$-invariant monotone coupling of the corresponding
determinantal probability measures witnessing the stochastic domination
$\P^{Q_1} \dom \P^{Q_2}$.
In particular, this applies to the wired and free uniform spanning forests,
which was known before only when $\gp$ is residually amenable.
In the case of spanning forests, we also give a second more explicit proof,
which has the advantage of showing an explicit way to create the free
uniform spanning forest as a limit over a sofic approximation.
Another consequence of our main result is to prove that all determinantal
probability measures $\P^Q$ as above are $\dbar$-limits of finitely
dependent processes. Thus, when $\gp$ is amenable, $\P^Q$ is
isomorphic to a Bernoulli shift, which was known before only when $\gp$ is
abelian.
We also prove analogous results for sofic unimodular random rooted graphs.
\end{abstract}

\maketitle

\tableofcontents

\bsection{Introduction}{s.intro}

The study of Bernoulli percolation and other random subgraphs of Cayley
graphs of non-amenable groups began to flourish in the mid 1990s.
Although the lack of averaging over F\o lner sequences was replaced by use
of the Mass-Transport Principle, and expansion of all finite sets was even
turned to advantage, coupling questions remained vexing:
In the amenable context, if two invariant probability
measures have the property that one stochastically dominates the other,
then there is an invariant coupling (joining) of the two measures that
witnesses the domination, which is called a monotone coupling.
Whether this holds on non-amenable groups remained open until resolved in
the negative by Mester \rref b.Mester:mono/.
Nonetheless, an invariant monotone coupling may still exist for certain
pairs of measures.
Indeed, the coupling question was originally motivated by the special case
of the so-called wired and free spanning forest measures, denoted $\wsf$
and $\fsf$.
Bowen \cite{bowen} showed that a monotone joining exists for $\wsf$ and
$\fsf$ on every residually amenable group.
We show that it exists for $\wsf$ and $\fsf$ on every sofic group, a wide
class of groups that no group is known not to belong to.
We show this as a consequence of a more general result: These spanning
forest measures are examples of determinantal probability measures, where,
as we review in the next section, for every positive contraction on
$\ell^2(W)$, $W$ being a denumerable set, there is associated a probability
measure $\P^Q$ on $2^W$. It is known \cite{Lyons:det,BBL:Rayleigh} that when $Q_1 \le
Q_2$, we have the stochastic domination $\P^{Q_1} \dom \P^{Q_2}$.
We show that
for any pair of
determinantal probability measures corresponding to equivariant
positive contractions
$Q_1 \le Q_2$ on $\ell^2$ of any sofic group, there is not only a monotone
coupling of $\P^{Q_1}$ with $\P^{Q_2}$, but in fact one that is
$\gp$-invariant.
Determinantal probability measures and point processes are a class of
measures that appear in a variety of contexts, including several areas of
mathematics as well as physics and machine learning. See, e.g., \cite{
Macchi, Soshnikov:survey, Lyons:det, 
KulTas}.

This is our main result, obtained via some abstract considerations of
ultraproducts of tracial von Neumann algebras.
The proof proceeds along the following broad outline: 
Given $\gp$-equivariant positive contractions $Q_1 \le Q_2$
and a sofic approximation to a Cayley diagram of
$\gp$ by finite graphs $G_n$ labelled with the generators of $\gp$, we form
the metric ultraproduct of the sequence of von Neumann algebras associated
to the sequence $(G_n)_n$.
This ultraproduct allows us to approximate $Q_i$ by positive contractions
$Q_{1, n} \le Q_{2, n}$ on $\ell^2\big(\verts(G_n)\big)$.
We may then take a limit point of monotone couplings of $\P^{Q_{1, n}}$ with
$\P^{Q_{2, n}}$ to obtain a $\gp$-invariant monotone coupling of $\P^{Q_1}$
with $\P^{Q_2}$.
We remark that ultraproducts are not essential to our proofs, but they
allow us to make convenient statements about operators before deducing
corresponding consequences for determinantal probability measures.

We also describe some consequences for comparisons of return probabilities
of random walks in random environments.
While we are able to deduce a consequence for $\fsf$ that was not
previously known, our coupling is not sufficiently explicit to deduce
much more.
In the motivating case of $\wsf$ and $\fsf$, we also give a somewhat more
concrete way to obtain such a coupling.
In particular, this concrete approach yields a simple way to obtain the
$\fsf$ as a limit over a sofic approximation.
Namely, for $L \ge 0$, let $\CYCLE_L(G)$ denote the space spanned by the
cycles in $G$ of length at most $L$.
Write $\fsf_{G, L}$ for the determinantal probability measure corresponding to
the orthogonal projection onto $\CYCLE_L(G)^\perp$.
We show that if a Cayley graph $G$
is the random weak limit of $(G_n)_n$, then for $L(n) \to\infty$
sufficiently slowly,
the random weak limit of $\fsf_{G_n, L(n)}$ 
equals $\fsf_G$.
We also extend our result from the context of Cayley graphs to its
natural setting of unimodular random rooted networks.

Finally, we derive
a consequence of our monotone joining result for the ergodic
theory of group actions.
Let $\gp$ be a countable group and $X$ and $Y$ be two sets on which $\gp$
acts. A map
$\phi \colon X \to Y$ is called \dfnterm{$\gp$-equivariant} if $\phi$
intertwines the actions of $\gp$: 
$$
\phi(\gpe x) = \gpe \big(\phi(x)\big) \qquad (\gpe \in \gp,\, x \in X)
\,.
$$
If $X$ and $Y$ are both measurable spaces, then a $\gp$-equivariant
measurable $\phi$ is called a \dfnterm{$\gp$-factor}.
Let $\mu$ be a measure on $X$. If $\phi$ is a $\gp$-factor, then the
push-forward measure $\phi_* \mu$ is called a \dfnterm{$\gp$-factor of $\mu$}.
The measure $\mu$ is \dfnterm{$\gp$-invariant} if 
$$
\mu(\gpe B) = \mu(B) \qquad (\gpe \in \gp,\, B \subseteq X \hbox{
measurable}).
$$
If $\nu$ is a measure on $Y$, then a $\nu$-a.e.-invertible $\gp$-factor
$\phi$ such that $\nu = \phi_* \mu$ is called an \dfnterm{isomorphism} from
$(X, \mu, \gp)$ to $(Y, \nu, \gp)$.
We are interested in the case where $X$ and $Y$ are product spaces
of the form $A^\gp$ or, more generally, $A^W$, where $A$ is a measurable
space and $W$ is a countable set on which $\gp$ acts. In such a case,
$\gp$ acts on $A^W$ by 
$$
\big(\gpe \omega\big)(x) 
:=
\omega( \gpe^{-1} x)  \qquad (\omega \in A^W,\, x \in  W,\, \gpe
\in \gp)
\,.
$$
If $\lambda$ is a probability measure on $A$ and
$\mu$ is the product measure $\lambda^W$, then we call the action of $\gp$
on $A^W$ a \dfnterm{Bernoulli shift}.
Ornstein \cite{Orn:book} proved a number of fundamental results about
Bernoulli shifts for $\gp = \Z$ that were extended in work with Weiss
\cite{OrnW:amen} to amenable groups.
In particular, they showed that factors of Bernoulli shifts are isomorphic
to Bernoulli shifts.
On the other hand, in the non-amenable setting, Popa
\rref b.Popa/ gave an example of a factor of a Bernoulli shift that is not
isomorphic to a Bernoulli shift.
More generally, it is not well understood which actions are factors of
Bernoulli shifts when $\gp$ is not amenable.
The utility of factors of Bernoulli shifts
on non-amenable groups has
been shown in various ways; for example, see 
\cite{Popa,
ChifanIoana,
Houdayer,
Lyons:fixed,
AW:Bernoulli,
Kun:Lip,
Lyons:fiid}.

It is easy to see that every factor of a Bernoulli shift is a $\dbar$-limit of
finitely dependent processes, by approximating the factor with a block factor.
When $\gp$ is amenable, the converse is true: every $\dbar$-limit of finitely
dependent processes is a factor of a Bernoulli shift.
Although we do not know whether 
determinantal probability measures on $2^\gp$ arising from equivariant
positive contractions are factors of Bernoulli shifts,
we show here that they are $\dbar$-limits of finitely dependent
$\gp$-invariant probability measures on $2^\gp$
provided $\gp$ is
sofic; in particular, when $\gp$ is amenable, these determinantal measures
are isomorphic to Bernoulli shifts, a result that was first shown for
abelian $\gp$ by \rref b.LS:dyn/.

In \rref s.def/, we give the relevant background on determinantal measures,
including the motivating case of spanning forests.
We discuss various basic notions related to groups, their Cayley graphs, 
their von Neumann algebras, and soficity in \rref s.cay/.
A review of ultraproducts and some new results we need is in \rref
s.proofs/.
These tools then lead quickly to a proof of our main result in \rref
s.existence/.
Following this, \rref s.approxim/ gives the alternative proof for the
spanning forest measures.
Consequences of an invariant coupling are discussed in \rref s.conseq/, including the
definitions of the $\dbar$-metric and finitely dependent processes.
Tools needed for the extension to unimodular random rooted networks are
given in \rref s.unimodular/.
In fact, this generalization has the advantage that the case of measures on
subsets of edges, which had to be treated separately and somewhat
cumbersomely in earlier sections, here can be deduced as merely a special
case.
After these tools are developed, we again have a short proof of the
existence of unimodular (sofic) couplings in \rref s.sofic-couple/.

\bsection{Determinantal probability measures}{s.def}

A determinantal probability measure is one whose elementary
cylinder probabilities are given by determinants.
More specifically, suppose that $E$ is a finite or countable set and that
$Q$ is an $E \times E$ matrix.
For a subset $A \subset E$, let $Q\restrict A$ denote the submatrix of
$Q$ whose rows and columns are indexed by $A$.
If $\qba$ is a random subset of $E$ with the property that for all finite $A
\subset E$, we have 
\rlabel e.DPM
{\P[A \subset \qba] = \det (Q\restrict A)
\,, }
then we call $\P$ a \dfnterm{determinantal probability measure}.
The inclusion-exclusion principle in combination with \eqref{e.DPM}
determines the probability of each elementary
cylinder event.
Therefore,
for every $Q$, there is at most one probability
measure satisfying \eqref{e.DPM}.
Conversely, it is known (see, e.g., \rref b.Lyons:det/) that there is a
determinantal probability measure corresponding to $Q$ if $Q$ is the matrix of
a \dfnterm{positive contraction} on $\ell^2(E)$ (in the standard
orthonormal basis), which means that for all $u \in \ell^2(E)$, we have
$0 \le \iprod{Q u, u} \le \iprod{u, u}$.

We identify a subset of $E$ with an element of $\{0, 1\}^E = 2^E$ in the
usual way.

An event $\mathcal A \subseteq 2^E$ is called \dfnterm{increasing} if for
all $A \in  \mathcal A$ and all $e \in  E$, we have $A \cup \{e\} \in
\mathcal A$.
Given two probability measures $\P^1$, $\P^2$ on $2^E$, we say that
\dfnterm{$\P^2$ stochastically dominates $\P^1$} and write $\P^1 \dom \P^2$ if
for all increasing events $\mathcal A$, we have $\P^1(\mathcal A) \le
\P^2(\mathcal A)$.
A \dfnterm{coupling} of two probability measures $\P^1$, $\P^2$ on
$2^E$ is a probability measure $\mu$
on $2^E \times 2^E$ whose coordinate projections are $\P^1$, $\P^2$.
A coupling $\mu$ is called
\dfnterm{monotone} if
$$
\mu\big\{(A_1, A_2) \st A_1 \subset A_2\big\} = 1
\,.
$$
By Strassen's theorem \cite{Strassen}, stochastic
domination $\P^1 \preccurlyeq \P^2$ is equivalent to the existence of a
monotone coupling of $\P^1$ and $\P^2$.
However, even if $\P^1$ and $\P^2$ are $\gp$-invariant probability measures
on $2^\gp$ with $\P^1 \dom \P^2$, it does not follow that there is a
$\gp$-invariant monotone coupling of $\P^1$ and $\P^2$; see \rref
b.Mester:mono/.
We need the following theorem; see \rref b.Lyons:det/ 
and \rref b.BBL:Rayleigh/.
\procl t.dominate
Let $E$ be finite and let
$Q_1 \le Q_2$ be positive contractions of $\ell^2(E)$. Then $\P^{Q_1} \dom
\P^{Q_2}$.
\endprocl
The most well-known example of a (nontrivial discrete) determinantal
probability measure is that where $\qba$ is a uniformly chosen random spanning
tree of a finite connected graph $G = (\vertex, \edge)$ with $E := \edge$.
In this case, $Q$ is the \dfnterm{transfer current matrix} $Y$,
which is defined as follows.
Orient the edges of $G$ arbitrarily.
Regard $G$ as an electrical network with each edge having unit
conductance.  Then $Y(e, f)$ is the amount of current flowing along the edge
$f$ when a battery is hooked up between the endpoints of $e$ of such
voltage that in the network as a whole, unit current flows from the tail of
$e$ to the head of $e$.
The fact that \eqref{e.DPM} holds for the uniform spanning tree is due to
\rref b.BurPem/ and is called the Transfer Current Theorem.
The case with $|A| = 1$ was shown much earlier by \rref b.Kirchhoff/, while
the case with $|A| = 2$ was first shown by \rref b.BSST/.
Write $\ust_G$ for the uniform spanning tree measure on $G$.

The study of the analogue on an infinite graph of a uniform spanning tree
was begun by \rref b.Pemantle:ust/ at the suggestion of Lyons.
Pemantle showed
that if an infinite connected graph $G$ is exhausted by a sequence of finite
connected subgraphs $G_n$, then the weak limit of the uniform spanning tree
measures $\ust_{G_n}$ on $G_n$ exists. 
However, it may happen that the limit measure is not supported on trees,
but on forests.
This limit measure is now called the \dfnterm{free uniform spanning forest} on
$G$, denoted $\fsf_G$. Considerations of electrical networks play the
dominant role in the proof of existence of the limit.
If $G$ is itself a tree, then this measure is
trivial, namely, it is concentrated on $\{G\}$. Therefore, \rref b.Hag:rcust/
introduced another limit that had been considered more implicitly by \rref
b.Pemantle:ust/ on $\Z^d$, namely, the weak limit of the uniform
spanning tree measures on $G_n^*$, where $G_n^*$ is the graph $G_n$ with
its boundary identified to a single vertex. As \rref b.Pemantle:ust/ showed,
this limit also always exists
on any graph and is now called the \dfnterm{wired uniform spanning forest},
denoted $\wsf_G$.  

In many cases, the free and the wired limits are the same. In particular,
this is the case on all euclidean lattices such as $\Z^d$.  The general
question of when the free and wired uniform spanning forest measures are
the same turns out to be quite interesting: The measures are the same iff
there are no nonconstant harmonic Dirichlet functions on $G$ (see
\BLPSusf). For a Cayley graph of a group $\Gamma$, this condition is
equivalent to the non-vanishing of the first $\ell^2$-Betti number of the
group, i.e., $\beta_1^{(2)}(\Gamma) \neq 0$. This important property of a
group and its implications have been studied extensively; see, for
example, \cite{bekka, petthom}. 

In the paper \BLPSusf, it was noted that the Transfer Current Theorem
extends to the free and wired spanning forests if one uses the free and
wired currents, respectively. 
To explain this, note that the orthocomplement of the row space $\STAR(G)$ of
the vertex-edge incidence matrix of a finite graph $G$ is the kernel, denoted
$\CYCLE(G)$, of the matrix.
We call $\STAR(G)$ the \dfnterm{star space} of $G$ and $\CYCLE(G)$ the
\dfnterm{cycle
space} of $G$.
For an infinite graph $G = (\vertex, \edges)$ exhausted by finite subgraphs
$G_n$, we let $\STAR(G)$ be the closure of $\bigcup \STAR(G_n^*)$ and
$\CYCLE(G)$ be the closure of $\bigcup \CYCLE(G_n)$, where we take the closure
in $\ell^2(\edges)$.
Then $\wsf_G$ is the determinantal probability measure corresponding to the
projection $P_{\STAR(G)}$, while $\fsf_G$ is the determinantal probability measure
corresponding to $P_{\CYCLE(G)}^\perp := P_{\CYCLE(G)^\perp}$.
In particular, $\wsf_G = \fsf_G$ iff $\STAR(G) = \CYCLE(G)^\perp$.

While the wired spanning forest is quite well understood,
the free spanning forest measure is in general poorly understood. 
A more detailed summary of uniform spanning forest measures can be found in
\rref b.Lyons:bird/.

If $E$ is infinite and $Q_n$ are positive contractions on $\ell^2(E)$
that tend to $Q$ in the weak operator topology (i.e., in each matrix
entry with respect to the standard orthonormal basis, we have
convergence), then clearly $\P^{Q_n}$ tend to $\P^Q$ weak*.  Since
$\STAR(G_n^*) \subset \CYCLE(G_n)^\perp$, it follows from \rref
t.dominate/ that $\wsf_G \dom \fsf_G$, though this was proved first by other
means in \BLPSusf.

\bsection{Cayley graphs and Cayley diagrams}{s.cay}

For any group $\gp$ acting on a set $X$, if the contraction $Q$ on
$\ell^2(X)$ is $\gp$-equivariant, then $\P^Q$ is $\gp$-invariant.

\begin{dfn} Let $\gp$ be a group. If $S$ is a set of elements of\/ $\gp$, we
write $S^{-1} := \{s^{-1} \st s \in  S\}$. If $S = S^{-1}$ is such that
the smallest subgroup of\/ $\gp$ that contains all elements of $S$ is $\gp$
itself, then the corresponding \dfnterm{Cayley graph} of\/ $\gp$ with respect
to $S$ is the undirected graph whose vertex set is $\gp$ and whose edge set
is $[\gpe, \gpe s]$ for $\gpe \in \gp$ and $s \in S$.
The \dfnterm{Cayley diagram} ${\rm Cay}(\Gamma,S)$ contains more
information and is a labelled oriented graph; namely, it is the graph whose
vertex set is $\gp$ and whose edge set is $(\gpe, \gpe s)$ for $\gpe \in
\gp$ and $s \in S$, where that edge is labelled $s$.
Thus, for each unoriented edge of the Cayley graph, there are two oriented
edges of the Cayley diagram, with inverse labels.
We shall always assume that $S = S^{-1}$.
An \dfnterm{$S$-labelled graph} is a graph each of whose edges is assigned
a label in $S$.
\end{dfn}

A \dfnterm{rooted graph} $(G, o)$ is a graph $G$ with a distinguished
vertex $o$ of $G$. We denote by $[(G, o)]$ the class of rooted graphs that
are isomorphic to $G$ via an isomorphism that preserves the root. If $G$ is
labelled or oriented, then we also require the isomorphism to preserve that
extra structure. Generally we are interested in isomorphism classes and
shall, after the first uses, drop the square brackets in the notation and
thus not distinguish notationally between a rooted graph and its
isomorphism class.

Given a vertex $v$ in a (possibly labelled and oriented) graph $G$ and $r
\ge 0$, write $B(v, r; G)$ for the (possibly labelled and oriented) graph
induced by $G$ on the vertices within distance $r$ of $v$, with $v$ as 
the root.  If $G$ is finite, let $\nu_{G, r}$
denote the law of $[B(v, r; G)]$ when $v$ is chosen uniformly at random.
If $G$ is a Cayley graph with identity element $o$ and $G_n$ are finite
undirected graphs such that 
for every $r > 0$, the laws $\nu_{G_n, r}$ tend to
$\delta_{[B(o, r; G)]}$, then we say that the \dfnterm{random weak limit} of
$(G_n)_n$ is $G$.

We say that $\gp$ is \dfnterm{sofic} if there exists a sequence $(G_n)_{n
\ge 1}$ of finite oriented graphs whose edges are labelled by elements of
$S$ such that for every $r > 0$, the laws $\nu_{G_n, r}$ tend to
$\delta_{[B(o, r; G)]}$, where $G$ is the Cayley diagram of $\gp$.  In this
case, we say that ${\rm Cay}(\Gamma,S)$ is a limit of finite $S$-labelled
graphs. It is well known and easy to see that if ${\rm Cay}(\Gamma,S)$ is
a limit of finite $S$-labelled graphs, then the same holds for any other
finite generating set of elements $S' \subset \Gamma$ \cite{Weiss}.

\begin{dfn}
A $S$-labelled oriented graph $G$ is called an \dfnterm{$S$-labelled
Schreier graph} if
for each vertex in $G$ and each $s \in S$, there is precisely one incoming edge and one outgoing edge with the label $s$.
\end{dfn}

Let $\bF_{S}$ denote the free group on the set $S$, i.e.,
$\bF_{S}$ consists of formal products of letters $s^{\pm}$ with
$s \in S$. We denote the neutral element of $\bF_S$ by $\varnothing$. Each $S$-labelled Schreier graph is equipped with a natural
action of $\bF_S$. The action of $s \in S$ on $v \in \vertex$
yields the unique vertex $v' \in \vertex$ such that there exists an
$s$-labelled oriented edge $(v,v')$. We shall denote this vertex by
$v.s$; and hence consider this action as a right action on $\vertex$.
More generally, for $v \in \vertex$ and $w \in \bF_S$, $v.w$
denotes the vertex which is obtained by an oriented walk following the
labels determined by $w \in \bF_S$. 
Likewise, given an edge $e = (v, v.s)$ and $w \in \FS$, write $e.w := (v.w, v.w.s)$.
For a set $A \subset \FS$, write $v.A := \{v.w \st w \in A\}$ and $e.A :=
\{e.w \st w \in A\}$.

The following was proved in slightly different language by
\cite{ES:direct}. 

\procl l.def-sofic
Let\/ $\Gamma$ be a group and $S$ a generating set of elements of\/
$\Gamma$. The group $\Gamma$ is sofic if and only if\/ ${\rm Cay}(\Gamma,S)$ is a limit of finite $S$-labelled Schreier graphs.
\endprocl

Suppose that $\theta$ is a probability measure on $\{0, 1\}^\vertex$ or
$\{0, 1\}^\edges$ for a graph $G = (\vertex, \edge)$, regarded as giving
random subsets (of either vertices or edges) by using $\{0, 1\}$-valued
marks.
For a vertex $v$ and $r \ge 0$, denote the restriction of $\theta$ to $B(v,
r; G)$ by $\theta(v, r)$.
Write $[\theta(v, r)]$ for the probability measure induced on $[B(v, r;
G)]$.
If $G$ is finite, let $\theta(G, r) := \frac1{|\vertex|}\sum_{v \in
\vertex} [\theta(v, r)]$.
In other words, if $H$ is a rooted $\{0, 1\}$-marked graph of radius at
most $r$, then with $\cong$ denoting isomorphism of rooted marked graphs,
$\theta(G, r)(H) = \frac1{|\vertex|}\sum_{v \in \vertex} \sum_{H' \cong
H} \theta(v, r)(H')$. 
Suppose that $G = (\vertex, \edge)$ is the random weak limit of finite 
graphs $G_n = (\vertex_n, \edge_n)$, that $\theta_n$ are probability
measures on $\{0, 1\}^{\vertex_n}$ or $\{0, 1\}^{\edges_n}$, and that
$\theta$ is a probability measure on $\{0, 1\}^\vertex$ or $\{0,
1\}^\edge$, respectively. 
If the weak* limit of $\theta_n({G_n}, r)$ equals $\theta(o, r)$ for
all $r \ge 0$, then we say that $\theta$ is the \dfnterm{random weak limit}
of $(\theta_n)_n$.

Let $\Gamma$ be a group and $S$ be a finite generating set of elements of
$\Gamma$. There is a natural action of $\Gamma$ on ${\rm Cay}(\Gamma,S)
=(\gp,\edge)$ that we denote by $\lambda \colon \Gamma \to {\rm
Aut}\big({\rm Cay}(\Gamma,S)\big)$. It is defined by the translations
$\lambda(g)(h) = gh$ for $g,h \in \gp$. The vertex set $\gp$ of ${\rm
Cay}(\Gamma,S)$ admits another $\Gamma$ action, which is given by the
formula $\rho(g)(h) = hg^{-1}$ for $g,h \in \Gamma$. Both actions extend to
unitary representations of $\Gamma$ on the Hilbert space $\ell^2 \Gamma$,
which we also denote by $\lambda$ and $\rho$. We denote the natural
orthonormal basis of $\ell^2 \Gamma$ by $\{\delta_\gamma \st \gamma \in
\Gamma \}$. Since the multiplication in $\Gamma$ is associative, these
actions commute. We define the (right) \dfnterm{group von Neumann
algebra} of $\Gamma$ to be the algebra of $\gp$-equivariant operators,
$$R(\Gamma) := \big\{ T \in B(\ell^2 \Gamma) \st \forall g \in \Gamma \colon
\lambda(g)T=T\lambda(g) \big\}\,,$$
and note that $\rho(\Gamma) \subset R(\Gamma)$. We use $\rho$ also to
denote the linearization $\rho \colon \bbC \Gamma \to R(\Gamma)$, which is
a $*$-homomorphism, where $\bbC \Gamma$ denotes the complex group ring,
which carries the natural involution that sends $g$ to $g^{-1}$ and is
complex conjugation on the coefficients. The group von Neumann
algebra $R(\Gamma)$ comes with a natural involution $T \mapsto T^*$ (the
adjoint map) and a \dfnterm{trace}
given by the formula 
$$
\tau(T) := \langle T \delta_o ,\delta_o \rangle
\,,
$$
where $o$ is the identity element of $\gp$. The functional $\tau \colon
R(\Gamma) \to \bbC$ leads to the definition of \dfnterm{spectral measure}:
Let $T \in R(\Gamma)$ be a self-adjoint operator. There exists a unique
probability measure $\mu_T$ on the interval $\big[{-}\|T\|,\|T\|\big] \subset \bbR$ such that
$$\forall n \in \bbN\quad \tau(T^n) = \int t^n \ d \mu_T(t)\, .$$

Note that the action of $\Gamma$ on $\edge$ is isomorphic to the natural left
action of $\Gamma$ on $\Gamma \times S$, where we identify a pair $(\gamma,s)
\in \Gamma \times S$ with the edge $(\gamma,\gamma s) \in \edge$. Abusing
notation, we denote the natural unitary action of $\Gamma$ on $\ell^2 \edge$
also by $\lambda$.
We define the von Neumann algebra of the Cayley diagram ${\rm Cay}(\Gamma,S)$ to be
$$R(\Gamma,S) := \big\{T \in B(\ell^2 \edge) \st \forall g \in \Gamma
\colon \lambda(g)T=T\lambda(g) \big\}\,.$$

\begin{lem} \label{block} There is a natural identification
$$R(\Gamma,S)=M_S\big(R(\Gamma)\big)\,,$$ where $M_S(Z)$ denotes the $S
\times S$-matrices over a ring $Z$.
\end{lem}
\begin{proof}
The isomorphism is described as follows. For $s \in S$, we denote the
orthogonal projection onto $\ell^2(\{e \in \edge \st e = (\gamma,\gamma s),
\gamma \in \Gamma \})$ by $p_s$. Clearly, $p_s$ is $\lambda$-equivariant
and there is a natural $\Gamma$-equivariant identification of $\{e \in
\edge \st e = (\gamma,\gamma s), \gamma \in \Gamma \}$ with $\Gamma$.
Indeed, since ${\rm Cay}(\Gamma,S)$ is a Schreier graph, there is precisely
one edge starting at $\gpe \in \Gamma$ with label $s \in S$.
Let now $s,t \in S$.
With a slight abuse of notation, every operator $T \in R(\Gamma,S)$ determines
a matrix $(p_s T p_t)_{s,t \in S}$ of elements in $R(\Gamma)$.
It is easily checked that this identification preserves all the structure.
\end{proof}

Note that the von Neumann algebra $R(\Gamma,S)$ also comes equipped with
a natural trace $\tau_S \colon R(\Gamma,S) \to \bbC$ given by $\tau_S(T)
:= \sum_{s \in S} \tau(p_s T p_s)$ with $p_s$ being as defined in the
proof of Lemma \ref{block}; we call $\tau_S$ the \dfnterm{natural
extension} of $\tau$.

Let now $G=(\vertex,\edge)$ be an $S$-labelled Schreier graph. We
consider the algebra $B(\ell^2 \vertex)$ of bounded operators on the
Hilbert space $\ell^2 \vertex$ and note that we have a natural
identification $B(\ell^2 \edge) = M_S(B(\ell^2 \vertex))$ as in the proof
of the previous lemma. If $\vertex$ is finite, then $B(\ell^2 \vertex)$
carries a natural normalized trace $$\tr_\vertex(T) :=
\frac{1}{|\vertex| }\sum_{v \in \vertex} \langle T \delta_v , \delta_v
\rangle\,.$$ Similarly, there is then a natural trace $\tr_{\edge}$ on the
algebra $B(\ell^2 \edge)$ when $\edges$ is finite.  Every element $w \in
\bF_S$ determines a
unitary operator $\pi(w)$ on $\ell^2(\vertex)$ for each $S$-labelled
Schreier graph
$G=(\vertex,\edge)$. It is given by linearity and the formula
\begin{equation} \label{defpi}
\forall w \in \bF_S \ \ \forall v \in \vertex \quad
\pi(w)(\delta_v) := \delta_{v.w^{-1}}\,.
\end{equation}
By taking linear combinations, every element in the complex group ring $a
\in \bbC \bF_S$ determines a bounded operator $\rho_G(a)$ in
$B(\ell^2 \vertex)$. Note that if $G= {\rm Cay}(\Gamma,S)$, then
$\rho_G(a)=\rho(\pi(a))$, where $\pi \colon \bbC \bF_S \to \bbC
\Gamma$ denotes the natural homomorphism of group rings extending the map
from $S$ to its image in $\Gamma$.

The following lemma is well known \cite{elekszabo}; let us include a proof for convenience.

\begin{lem} \label{lem:sofic}
Let\/ $\gp$ be a sofic group and $S$ be a finite generating set of elements
in $\gp$. Let $(G_n)_n$ be a sequence of finite $S$-labelled Schreier
graphs whose limit is $G:={\rm Cay}(\gp,S)$. Then for each $a \in \bbC
\bF_{S}$, we have
$$\tau(\rho_G(a)) = \lim_{n \to \infty} \tr_{\vertex_n}(\rho_{G_n}(a))\,.$$
\end{lem}
\begin{proof} It is enough to prove the statement for $a = w \in {\mathbf
F}_S$. Then the left-hand side is $1$ or $0$ depending on whether
$\pi(w)= o$ in $\Gamma$ or not.  The right-hand side equals
$$\lim_{n \to \infty}\frac{1}{|\vertex_n| }\sum_{v \in \vertex_n} \langle
\rho_{G_n}(w) \delta_v , \delta_v \rangle$$ and thus measures the
fraction of vertices that are fixed by the action of $\rho_{G_n}(w)$. If
$w^{-1} = s_1^{\varepsilon_1} \cdots s_n^{\varepsilon_n}$ with
$s_1,\dots,s_n \in S$ and $\varepsilon_1,\dots,\varepsilon_n \in \{\pm
1\}$, then this is the fraction of vertices $v \in \vertex_n$ for which an
oriented walk with labels $s_1^{\varepsilon_1}, \dots, s_n^{\varepsilon_n}$
ends at $v \in \vertex_n$, i.e., if $v.w^{-1}=v$ or not. It is clear that this fraction converges to $1$
or $0$ depending on whether the corresponding walk in $G$ starting at $o
\in \Gamma$ ends at $o \in \Gamma$ or not. It ends at $o$ if and only if
$\pi(w) = o$ in $\Gamma$. Hence, the left- and the right-hand sides are
equal. This finishes the proof.
\end{proof}

\bsection{Tracial von Neumann algebras and their ultraproducts}{s.proofs}

Before we introduce the metric ultraproduct of tracial von Neumann
algebras, we shall recall some basic facts about von Neumann algebras. A
\dfnterm{von Neumann algebra} is defined to be a weakly closed unital
$*$-subalgebra of the space $B(H)$ of bounded linear operators on some
Hilbert space $H$. 
Note that the weak topology here refers to the weak operator topology (WOT).

If $K \subset B(H)$ is a self-adjoint subset, then we define the
\dfnterm{commutant} of $K$ as $K' := \{T \in B(H) \st \forall S \in K \
ST=TS \}$. It is easy to see that $K'$ is a von Neumann algebra. Most von Neumann algebras arise naturally as the
commutant of some explicit set of operators.  For example, we defined
$R(\Gamma)= \{\lambda(\gamma) \in B(\ell^2 \Gamma) \st \gamma \in \Gamma
\}'$. Von Neumann's Double Commutant Theorem says that this construction
exhausts all von Neumann algebras; more specifically, $M \subset B(H)$ is a
von Neumann algebra if and only if $M=M^*$ and $M=M''$.

Since any weakly closed algebra is also norm closed in $B(H)$, a von
Neumann algebra is in particular a $C^*$-algebra. Like $C^*$-algebras,
von Neumann algebras admit an abstract characterization. A $C^*$-algebra
$M$ is $*$-isomorphic to a von Neumann algebra if and only if as a
Banach space, $M$ is a dual Banach space; in this case, the
pre-dual $M_*$ is unique up to isometry. In case $M=B(H)$, the pre-dual
$M_*$ is the
Banach space of trace-class operators on $H$ with the natural pairing
between $M$ and $M_*$ given by the trace. It is well known that the weak*
topology (seeing $B(H)$ as a dual Banach space) is nothing but the
so-called ultra-weak topology. Moreover, if $M$ is a weakly closed
sub-algebra of $B(H)$, then the natural weak* topology on $M = (M_*)^*$
can be identified with the ultra-weak topology inherited from $B(H)$.
 
Following Dedekind's characterization (or rather definition) of an infinite set, a projection
(which always means ``orthogonal projection" for us) $P \in M$ in a von
Neumann algebra is called \dfnterm{infinite} if it is equivalent in $M$ to a
subprojection of itself, i.e., if
there exists a partial isometry  $U \in M$ such that
$U^*U=P$ and $UU^*$ is a proper subprojection of $P$. A projection
that is not infinite is called \dfnterm{finite}.  For our purposes, only
\dfnterm{finite} von Neumann algebras matter, and these are (by definition)
the ones in which every projection is finite. Note that $B(\ell^2 \vertex)$ is finite if and only if $\vertex$ is finite.

A certificate for finiteness of a von Neumann algebra $M$ is the
existence of a positive, faithful trace $\tau \colon M \to \bbC$.  Here, a
(linear) functional $\tau \colon M \to \bbC$ is called \dfnterm{positive}
if $\tau(T^*T) \geq 0$. A positive functional is called \dfnterm{faithful} if $\tau(T^*T)=0$ implies
$T=0$. The functional $\tau$ is called a \dfnterm{trace} if
$\tau(TS)=\tau(ST)$ for all $S,T \in M$. For the sake of getting used to
the definitions, let us convince ourselves that a von Neumann algebra $M$
with a positive and faithful trace $\tau \colon M \to \bbC$ is finite.
Let $P$ be a projection and $U$ be a partial isometry with $U^*U=P$. 
Suppose that $UU^*$ is a
subprojection of $P$. Then $Q:=P-UU^*$ is also a projection and hence
$Q=Q^2=Q^*Q$.  However, $\tau(Q^*Q)=\tau(Q) = \tau(U^*U-UU^*) = 0$ and hence $Q=0$. This shows that $UU^*=P$. 
Hence, $M$ is a finite von Neumann algebra. 

A functional $\tau \colon M \to \bbC$ is called
\dfnterm{normal} if it is ultra-weakly continuous; natural examples are
sums of so-called vector states $\tau(T):= \sum_i \langle T \xi_i,\xi_i
\rangle$ for $\xi_1,\dots,\xi_n \in H$. The functional 
is called \dfnterm{unital} if
$\tau(1)=1$. 
We call a pair $(M,\tau)$ a \dfnterm{tracial von Neumann
algebra} if $M$ is a von Neumann algebra and $\tau \colon M \to \bbC$ is
a normal, positive, faithful, and unital trace.  One can show that $\tau
\colon R(\Gamma) \to \bbC$ is a normal, positive, faithful, and unital
trace. Hence, $R(\Gamma)$ is a finite von Neumann algebra.

A $*$-homomorphism $\varphi \colon (N,\tau_N) \to (M,\tau_M)$ between
$*$-algebras equipped with traces is called \dfnterm{trace preserving} if
$\tau_N= \tau_M \circ \varphi$.  If $\tau \colon M \to \bbC$ is a normal,
positive, faithful, and unital trace, then the norm $T \mapsto
\tau(T^*T)^{1/2}$ determines the ultra-weak topology on norm-bounded sets.
In particular, this implies that if $(N,\tau_N)$ and $(M,\tau_M)$ are
tracial von Neumann algebras, $N_0 \subset (N,\tau_N)$ is a dense
$*$-subalgebra, and $\tilde \varphi \colon (N_0,\tau_N \restrict {N_0}) \to
(M,\tau_M)$ is a trace-preserving $*$-homomorphism, then $\tilde \varphi$
has a unique extension to a trace-preserving $*$-homomorphism $\varphi
\colon (N,\tau_N) \to (M,\tau_M)$.  For all these facts about von Neumann
algebras and more, see \cite{takesaki} and the references therein.

The following construction of metric ultraproducts of von Neumann algebras
preserves the class of tracial von Neumann algebras and plays a crucial
role in their study. 
\begin{dfn} \label{dfn:ultra}
Let $(M_n,\tau_n)$ be a sequence of tracial von Neumann algebras.  Let $\omega$ be a non-principal ultrafilter on $\bbN$. We consider 
$$\ell^{\infty}(\bbN,(M_n,\tau_n)_n) := \big\{ (T_n)_n \st T_n \in M_n, \ \sup
\{\|T_n\| \st n \in \bbN\}< \infty \big\}$$
and set
$$J_{\omega}:= \big\{ (T_n)_n \in \ell^{\infty}(\bbN,(M_n,\tau_n)_n) \st
\lim_{n \to \omega} \tau_n(T_n^*T_n) = 0 \big\}\,.$$
It is well known and easy to verify that $J_{\omega} \subset
\ell^{\infty}(\bbN,(M_n,\tau_n)_n)$ is a two-sided $*$-ideal.
The \dfnterm{metric ultraproduct} of $(M_n,\tau_n)$ is defined to be the
quotient
$$\prod_{n \to \omega} (M_{n},\tau_{n}) :=
\frac{\ell^{\infty}(\bbN,(M_n,\tau_n)_n)}{J_{\omega}} \quad \mbox{with
trace}
\quad \tau_{\omega}\big((T_n)_n+J_{\omega}\big) := \lim_{n \to \omega}
\tau_n(T_n)\,.$$
\end{dfn}

It has been shown by Connes \cite[Section I.3]{connes} that $(\prod_{n \to
\omega} (M_{n},\tau_{n}),\tau_{\omega})$ is again a von Neumann algebra
acting on a concrete Hilbert space ${}_\omega H$. This is not
straightforward, since---if $M_n \subset B(H_n)$---the ultraproduct
$\prod_{n \to \omega} (M_{n},\tau_{n})$ does {\em not} act on the Hilbert space arising as the ordinary
Banach space ultraproduct $H_{\omega}:= \prod_{n \to \omega} H_n$ of the
sequence of Hilbert spaces $(H_n)_n$. One can show that $\tau_{\omega}
\colon \prod_{n \to \omega} (M_{n},\tau_{n}) \to \bbC$ is a normal,
positive, faithful, and unital trace.

In the context of the previous definition, we shall say that a sequence
$(T_n)_n$ of operators with $T_n \in M_n$ \dfnterm{represents} some
operator $T$ in
the ultraproduct if $(T_n)_n \in \ell^{\infty}(\bbN,(M_n,\tau_n)_n)$ and
$T$ is its residue class modulo $J_{\omega}$. Since $J_{\omega}$ is a
two-sided $*$-ideal, we get: If $(T_n)_n$ represents $T$ and $(S_n)_n$
represents $S$, then $(T_n^*)_n$ represents $T^*$, $(T_n+S_n)_n$ represents
$T+S$, and $(T_nS_n)_n$ represents $TS$. The following lemma is only
slightly more involved. 
Recall that if $T$ is a self-adjoint operator and $f \colon \bbR \to \bbR$ is
continuous, then $f(T)$ is defined by using the spectral decomposition of
$T$.

\begin{lem} \label{lem:app}
Let $(T_n)_n$ be a sequence of self-adjoint operators in
$\ell^{\infty}(\bbN,(M_n,\tau_n)_n)$ representing some operator $T$ in the ultraproduct. Let $f \colon \bbR
\to \bbR$ be a continuous function. Then the sequence $(f(T_n))_n$ represents $f(T)$. Let $(S_n)_n$ be another sequence in $\ell^{\infty}(\bbN,(M_n,\tau_n)_n)$ of self-adjoint operators representing some operator $S$ in the ultraproduct. If $S_n \leq T_n$ for all $n \in \bbN$, then $S \leq T$.
\end{lem}

\begin{proof}
By the above, the first statement is clear for polynomial functions. It follows
for general continuous functions by the Stone-Weierstrass approximation
theorem. Note that $S_n \leq T_n$ if and only if $T_n - S_n = C_n^*C_n$
for some $C_n \in M_n$. Since $(S_n)_n, (T_n)_n \in
\ell^{\infty}(\bbN,(M_n,\tau_n)_n)$, we also have $(C_n)_n \in
\ell^{\infty}(\bbN,(M_n,\tau_n)_n)$. Let $C$ be the operator that is
represented by the sequence $(C_n)_n$. Then $T-S = C^*C$ and hence $S
\leq T$. This finishes the proof.
\end{proof}

\begin{remark} \label{specmeas}The definition of metric ultraproduct is made in such a way that it is
compatible with spectral measures in the following sense.
If $(T_n)_n$ is a sequence of
self-adjoint operators in $\ell^{\infty}(\bbN,(M_n,\tau_n)_n)$ that
represents some self-adjoint operator $T$ in the ultraproduct,
then by definition and the above remarks, $\lim_{n \to \omega}
\tau_n(T_n^k) = \tau_\omega(T^k)$ for all $k \ge 0$, whence
$\lim_{n \to \omega} \mu_{T_n} = \mu_T$ in the weak* topology.
\end{remark}

The following proposition summarizes and extends results from \cite{elekszabo}.

\begin{pro} \label{emb}
Let $\gp$ be a sofic group and $S$ be a finite generating set of elements
in $\gp$. Let $(G_n)_n$ be a sequence of finite $S$-labelled Schreier
graphs whose limit is $G:={\rm Cay}(\gp,S)$. There exist trace-preserving
embeddings
$$\iota \colon (R(\Gamma),\tau) \to \prod_{n \to \omega} (B(\ell^2
\vertex_n),\tr_{\vertex_n}) \quad \mbox{and} \quad \iota_S
\colon(R(\Gamma,S),\tau_S) \to \prod_{n \to \omega}(B(\ell^2
\edge_n),\tr_{\edge_n})\,.$$
Moreover, if $(T_n)_n \in \ell^{\infty}( \bbN, (B(\ell^2 \vertex_n),{\rm
tr}_{\vertex_n}))$ represents $\iota(T)$ for some $T \in R(\Gamma)$, then
\begin{equation} \label{eq1}
\lim_{n \to \omega} |\vertex_n|^{-1} \sum_{v \in \vertex_n}
|\iprod{T_n \delta_{v.\gamma}, \delta_{v.\gamma'}} - \iprod{T
\delta_{o.\gpe}, \delta_{o.\gpe'}}| = 0
\end{equation}
for all $\gpe, \gpe' \in \bF_S$. 
Similarly, if $(T_n)_n \in
\ell^{\infty}( \bbN, (B(\ell^2 \edge_n),\tr_{\edge_n}))$
represents $\iota_S(T)$ for some $T \in R(\Gamma,S)$, then
\begin{equation} \label{eq1b}
\lim_{n \to \omega} |\vertex_n|^{-1} \sum_{v \in \vertex_n}
|\iprod{T_n \delta_{(v.\gamma,v.\gamma s)},\delta_{(v. \gamma',v.\gamma's')}} - \iprod{T
\delta_{(o.\gamma,o.\gamma s)}, \delta_{(o.\gamma', o.\gamma's')}}| = 0
\end{equation}
for all $\gpe, \gpe' \in \bF_S$ and $s,s' \in S$.

\end{pro}

Embedding theorems like the preceding proposition were first proved by
Elek and Szab\'o; see, for example, \cite[Theorem 2]{elekszabo}. However, our
emphasis is on the special features of embeddings coming from a sofic
approximation. In particular, we are interested in Equations \eqref{eq1} and
\eqref{eq1b}, which are not just features of every trace-preserving
embedding.

\begin{proof}[Proof of Proposition \ref{emb}:]
For each $n \in \bbN$, we consider the $*$-homomorphism $$\rho_{G_n}
\colon \bbC \bF_S \to (B(\ell^2 \vertex_n),{\rm
tr_{\vertex_n}})\,.$$ This sequence induces a $*$-homomorphism
$$\rho_\omega \colon \bbC \bF_S \to \prod_{n \to \omega}
(B(\ell^2 \vertex_n),\tr_{\vertex_n})\,.$$ We denote the canonical
trace on $\prod_{n \to \omega} (B(\ell^2 \vertex_n),{\rm
tr}_{\vertex_n})$ by $\tr_{\omega} := \lim_{n \to \omega} {\rm
tr}_{\vertex_n}$. Since $(G_n)_n$ is a sofic approximation to $\Gamma$, we
obtain
$$\tau_{\omega}(\rho_\omega(w)) = \begin{cases} 1 & \mbox{if }\pi(w)=e \mbox{ in
$\Gamma$}, \\
0 & \mbox{if }\pi(w)\neq e \mbox{ in $\Gamma$}. \end{cases}$$
As the result depends only on $\pi(w)$,
this shows that $\rho_\omega$ descends to a $*$-homomorphism
$\rho_\omega \colon \bbC \Gamma \to \prod_{n \to \omega} (B(\ell^2
\vertex_n),\tr_{\vertex_n})$ that preserves the canonical trace on
$\bbC \Gamma$. Note also that $\rho(\bbC \Gamma) \subset R(\Gamma)$ is
weakly dense. Indeed, this is a standard fact and follows, for example,
from the Commutation Theorem \cite[Theorem 1 on page 80]{Dixmier}. If there
is no risk of confusion, we shall sometimes identify $\bbC\Gamma$ with its image in $R(\Gamma)$. 
It is a standard fact---see the remarks before Definition
\ref{dfn:ultra}---that $\rho_\omega$ has an extension $\iota$ to the group von Neumann $R(\Gamma)$ of $\Gamma$.

The second embedding is obtained after passing to matrix algebras
$M_S(\cbuldot)$ via the natural isomorphisms $M_S R(\Gamma) = R(\Gamma,S)$ and $M_S(B(\ell^2 \vertex_n)) = B(\ell^2 \edge)$. 

We shall now prove Equation \eqref{eq1}. Upon replacing $T$ by
$\rho_G(\gamma')T \rho_G(\gamma^{-1})$ and $T_n$ by
$\rho_{G_n}(\gamma')T_n \rho_{G_n}(\gamma^{-1})$, we may assume that
$\gamma=\gamma'=o$. We first study the case $T=0$. We need to show that 
\begin{equation} \label{eq2}
\lim_{n \to \omega} |\vertex_n|^{-1} \sum_{v \in \vertex_n}
|\iprod{T_n \delta_{v}, \delta_{v}}| = 0.
\end{equation}
We may write $T_n := \sum_{k=0}^3 i^k T^{(k)}_n$ with $T_n^{(k)}$ positive
for each $0 \leq k \leq 3$ and $n \in \bbN$. Here, e.g., $T_n^{(0)} -
T_n^{(2)} = \Re (T_n) := (T_n + T_n^*)/2$ are self-adjoint operators that
form a sequence representing $(T + T^*)/2 = 0$, and $T_n^{(0)} = (\Re
T_n)_+$ is a function of $\Re T_n$.
By Lemma \ref{lem:app} and the remarks preceding it, each
sequence $(T_n^{(k)})_n$ represents zero in the ultraproduct. Hence,
Equation \eqref{eq2} holds for $(T_n^{(k)})_n$ for $0 \leq k \leq 3$
because all scalar products are positive already, and hence \eqref{eq2}
holds for $(T_n)_n$
by the triangle inequality.

We conclude again from the triangle inequality that if Equation \eqref{eq1} holds for one sequence $(T_n)_n$ representing some operator $\iota(T)$ in the ultraproduct, then it holds for every sequence representing $\iota(T)$. Thus, whether or not Equation \eqref{eq1} holds is a property of $\iota(T) \in \prod_{n \to \omega} (B(\ell^2 \vertex_n),\tr_{\vertex_n})$ alone.

It is clear that Equation \eqref{eq1} holds for $\iota(\rho(\gamma))$ for
$\gamma \in \Gamma$, as follows directly from the fact that the limit of
$(G_n)_n$ is $G$. It
is also clear that the set of operators $T \in R(\Gamma)$ for which
Equation \eqref{eq1} holds is closed under addition and multiplication by
scalars. Thus, Equation \eqref{eq1} holds for the complex group ring
$\bbC \Gamma \subset R(\Gamma)$ and every ultrafilter. Moreover, our
construction shows that for each $T \in \bbC \Gamma$, there exists a
sequence $(T_n)_n \in \ell^{\infty}( \bbN, (B(\ell^2 \vertex_n),{\rm
tr_{\vertex_n}}))$ such that
$$\lim_{n \to \infty} |\vertex_n|^{-1} \sum_{v \in \vertex_n}
|\iprod{T_n \delta_{v}, \delta_{v}} - \iprod{T
\delta_o, \delta_o}| = 0
\,.$$
Moreover, $(T_n)_n$ represents $\iota(T)$ in the ultraproduct with respect to any ultrafilter.
Now, let $T \in R(\Gamma)$ and let us choose a sequence $(S_i)_i$ with
$S_i \in \bbC \Gamma$, $S_i \to T$ in the weak operator topology and $\sup_i \|S_i\| \leq
\|T\|$. For each $i$, let $(S_{i,n})_n$ be a sequence that represents
$\iota(S_i)$ in the ultraproduct. It is now clear that if the sequence
$(k_n)_n$ of natural numbers increases slowly enough, then $(S_{k_n,n})_n$
represents $\iota(T)$ in the ultraproduct and that Equation \eqref{eq1}
holds for this sequence. Since we have shown that Equation \eqref{eq1}
depends only on $\iota(T)$, it follows that Equation \eqref{eq1} holds for
any sequence $(T_n)_n$ that represents $\iota(T)$ in the ultraproduct.
This proves Equation \eqref{eq1} for every $T \in R(\Gamma)$.
The proof of the second equality is similar. This finishes the proof of the proposition.
\end{proof}



The following lemma shows that our previous result is useful to connect the approximation of some operator $\iota(T)$ by a sequence $(T_n)_n$ to an approximation of the associated determinantal probability measures.

\procl l.ultradetl
Let $\gp$ be a sofic group and $S$ be a finite generating set of elements
in $\gp$. Let $(G_n)_n$ be a sequence of finite $S$-labelled Schreier
graphs whose limit is $G:={\rm Cay}(\gp,S)$. Let $\iota$ and $\iota_S$
be trace-preserving embeddings as in Proposition \ref{emb}.
Let $T \in R\gp$ be such that $0 \leq T \leq I$ and suppose that $(T_n)_n$
represents $\iota(T)$ in the ultraproduct $\prod_{n \to \omega}
(B(\ell^2 \vertex_n),\tr_{\vertex_n})$ with $0 \leq T_n \leq I$ for
each $n \in \bbN$.
Then $\lim_{n \to \omega} \P^{T_n} = \P^T$ in the random weak topology,
in other words,
$$
\lim_{n \to \omega}
\frac{1}{|\vertex_n| }\sum_{v \in \vertex_n} 
\P^{T_n}[v.A \subset \qba]
=
\P^{T}[o.A \subset \qba]
$$
for all finite $A \subset \bF_S$.
Likewise, 
if $T \in R(\gp, S)$ is such that $0 \leq T \leq I$ and $(T_n)_n$
represents $\iota_S(T)$ in the ultraproduct $\prod_{n \to \omega}
(B(\ell^2 \edge_n),\tr_{\edge_n})$ with $0 \leq T_n \leq I$ for
each $n \in \bbN$,
then $\lim_{n \to \omega} \P^{T_n} = \P^T$ in the random weak topology;
in other words, given a finite set $A_s \subset {\mathbf F_S}$ for each
$s \in  S$, we have
$$
\lim_{n \to \omega}
\frac{1}{|\vertex_n| }\sum_{v \in \vertex_n} 
\P^{T_n}\Big[\bigcup_{s \in S} (v, v.s).A_s \subset \qba\Big]
=
\P^{T}\Big[\bigcup_{s \in S} (o, s).A_s \subset \qba\Big]
\,.
$$
\endprocl

\rproof
We prove the first statement, as the second is similar and just involves
more notation.

By definition of determinantal probability measures, the desired identity
is the same as
$$
\lim_{n \to \omega}
\frac{1}{|\vertex_n| }\sum_{v \in \vertex_n} 
\det (T_n \restrict v.A)
=
\det (T \restrict o.A)
\,.
$$
Because $G$ is the limit of $(G_n)_n$, we may assume that $|o.A| = |A|$,
i.e., that $\pi$ is injective on $A$.
Define 
$$
a_n(v, \gpe, \gpe') 
:=
|\iprod{T_n \delta_{v.\gamma}, \delta_{v.\gamma'}} - \iprod{T
\delta_\gpe, \delta_{\gpe'}}| 
\,.
$$
By Equation (\ref{eq1}), we have 
$$
\lim_{n \to \omega}
\frac{1}{|\vertex_n| }\sum_{v \in \vertex_n} 
\sum_{\gpe, \gpe' \in A}
a_n(v, \gpe, \gpe')
=
0
\,.
$$
Since $a_n(v, \gpe, \gpe') \le 2$, it follows that
$$
\lim_{n \to \omega}
\frac{1}{|\vertex_n| }\sum_{v \in \vertex_n} 
\prod_{\gpe' \in A} \sum_{\gpe \in A}
a_n(v, \gpe, \gpe')
=
0
\,,
$$
whence by Hadamard's inequality, 
$$
\lim_{n \to \omega}
\frac{1}{|\vertex_n| }\sum_{v \in \vertex_n} 
|\det (T_n \restrict v.A)
-
\det (T \restrict o.A)|
=
0
\,,
$$
which is stronger than the limit we desired.
\Qed

The following lemma is well known \cite{Douglas}, but we provide a proof
for the convenience of the reader, as it is an essential tool for our
proofs.

\begin{lem} \label{prev}
Let $H$ be a Hilbert space, $M \subseteq B(H)$ be a von Neumann algebra, and $S,T \in M$ with $0 \leq S \leq T \leq I$. Then there exists a positive contraction $C \in M$ with $T^{1/2}CT^{1/2}=S$.
\end{lem}
\begin{proof}
Let $H_0 \subset H$ be the closure of the image of $T$ and note that $H_0 =
\ker(T)^{\perp}$. Since the orthogonal projection onto $H_0$ is contained in $M$, we may assume without loss of generality that $H_0=H$. The
space $H_0':=\{T^{1/2}\xi \st \xi \in H \}$ is dense in $H_0$. We define
$D$ to be $D(T^{1/2}\xi) := S^{1/2}\xi$ on $H_0'$. Now, since $0 \leq S
\leq T$, it is easy to see that $D$ is well defined and extends to a
contraction. Indeed, $$\|D(T^{1/2}\xi)\|^2 = \langle
D(T^{1/2}\xi),D(T^{1/2}\xi) \rangle =\langle S^{1/2}\xi, S^{1/2}\xi \rangle
= \langle S \xi,\xi \rangle \leq \langle T \xi,\xi \rangle = \|T^{1/2}
\xi\|^2\,.$$
The inequality shows that $D$ is well defined; the whole
shows that $D$ is contractive on $H_0'$ and hence has a unique contractive
extension to $H$. It is also clear that $DT^{1/2} = S^{1/2}$ on $H$. For every operator $U \in M'$, we get
$$UDT^{1/2} = US^{1/2} = S^{1/2}U = DT^{1/2}U = DUT^{1/2}.$$ Hence, $UD=DU$ on the image of $T^{1/2}$, which we assume is dense in $H$. This implies $D \in M''$ and hence $D \in M$, since $M$ is a von Neumann algebra.
We can now set $C:=D^*D \geq 0$, note that $C \in M$, and that we have $T^{1/2} CT^{1/2} = T^{1/2} D^*DT^{1/2} = S^{1/2} S^{1/2} = S$. It is obvious that $C$ is also a contraction. This finishes the proof of the lemma.
\end{proof}

The previous lemma can now be used to show that self-adjoint operators $S,T$ in the ultra-product with $0\leq S \leq T \leq I$ can be approximated by sequences satisfying the same relation.

\begin{lem} \label{getlim}
Let $(M_n,\tau_n)$ be a sequence of tracial von Neumann algebras. Let $S,T
\in \prod_{n \to \omega} (M_{n},\tau_{n})$ be operators such that $0 \leq S
\leq T \leq I$. Then there exist sequences $(T_n)_n$ and $(S_n)_n$ with
$T_n,S_n \in M_n$ that represent $T$ and $S$ in the ultraproduct and
such that $0 \leq S_n \leq T_n \leq I$ for each $n \in \bbN$.
\end{lem}
\begin{proof}
First of all, by Lemma \ref{prev} there exists a positive contraction $C\in \prod_{n \to \omega} (M_{n},\tau_{n})$ such that $T^{1/2}C T^{1/2} = S$.
Let $(T_n)_n$ be some representative of $T$. By Lemma \ref{lem:app},
$\big((T_n^* T_n)^{1/2}\big)_n$ represents $(T^* T)^{1/2} = T$, so by
replacing $T_n$ with
$(T_n^*T_n)^{1/2}$, we may assume
that $T_n \geq 0$. Let $\epsilon_n := \mu_{T_n}\big((1,\infty)\big)$, where $\mu_{T_n}$
denotes the spectral measure of $T_n$. Since $T \leq I$, we
know that $\epsilon_n \to 0$ as $n \to \infty$. Let $Q_n$ be the spectral projection onto the
interval $[0,1]$; thus, $\tau_n(Q_n) = 1 - \epsilon_n \to 1$ as $n \to
\infty$. This implies that $(Q_n)_n$ represents $I$ in the ultraproduct.
Hence, we may replace $T_n$ by $Q_n T_n$ to obtain a sequence $(T_n)_n$
that represents $T$ and satisfies $0 \leq T_n \leq I$ for all $n \in
\bbN$. The same argument applies to $0 \leq C \leq I$ and we obtain a
sequence $(C_n)_n$ representing $C$ such that $0 \leq C_n \leq I$ for all $n \in \bbN$. Moreover, the sequence $(S_n)_n$ with $S_n:=T_n^{1/2} C_n T_n^{1/2}$ represents $S=T^{1/2}CT^{1/2} $. Now
$0 \leq C_n \leq I$ implies $0 \leq S_n \leq T_n$. This finishes the proof.
\end{proof}

\bsection{Existence of invariant monotone couplings}{s.existence}

We shall now use the previous results to show the existence of certain
$\Gamma$-invariant couplings between $\Gamma$-invariant determinantal
measures on $\Gamma$ itself, as well as on the edge set of any Cayley graph
of $\Gamma$. Recall that a positive contraction $Q$ in $R(\Gamma)$ leads to
a $\Gamma$-invariant determinantal measure on $\Gamma$, which we denote by
$\P^Q$. Note also that
$\P^{Q_2}$ stochastically dominates $\P^{Q_1}$ if $0 \leq Q_1
\leq Q_2 \leq I$. This is stated in \rref t.dominate/ in the finite case
and is shown to imply the same for the infinite case in \rref b.Lyons:det/.

\procl t.gencouple
Let $\gp$ be a sofic group and $S$ be a finite generating set of elements
in $\gp$. 
If $0 \le Q_1 \le Q_2 \le I$ in $R(\gp)$ or in $R(\gp, S)$, then there exists
a $\gp$-invariant (sofic) monotone coupling of\/ $\P^{Q_1}$ and $\P^{Q_2}$.
\endprocl

\rproof
Let $(G_n)_n$ be a sequence of finite $S$-labelled Schreier graphs whose
limit is $G:={\rm Cay}(\gp,S)$.
We prove the theorem for operators in $R(\gp)$, as the other case is
essentially identical.
By Proposition \ref{emb} and Lemma \ref{getlim}, there exist $0 \le S_n \le
T_n \le I$ in $B(\ell^2(\verts_n))$ so that $(S_n)_n$ and $(T_n)_n$
represent $\iota(Q_1)$ and $\iota(Q_2)$ in the ultraproduct.
Let $\mu_n$ be a monotone coupling of $\P^{S_n}$ with $\P^{T_n}$, as obtained from Strassen's theorem \cite{Strassen}. 
As explained in \cite[Example 10.3]{AL:urn}, the random weak limit of
$(\mu_n)_n$ (perhaps after
taking a subsequence) is a $\gp$-invariant coupling of $\P^{Q_1}$ and
$\P^{Q_2}$, which is necessarily monotone. However, we give another proof
in the framework we are using here.

Note that
$\mu_n$ is a probability measure on $2^{\vertex_n} \times 2^{\vertex_n}$.
In general, given a set $V$, let $V'$ be a disjoint copy of $V$ with
bijection $\phi \colon V \to V'$ and
identify elements of $2^V \times 2^V$ with subsets of $V \cup V'$.
Thus, elementary cylinder events are identified with events of the form
$\mathcal A = \{A \subset V \cup V' \st A \supset A_1,\, A \cap A_2 =
\varnothing\}$ for finite $A_1, A_2 \subset V \cup V'$.
Let $A_1,A_2 \subset \gp \cup \gp'$.
Let $\mathcal A := \{A \subset \gp \cup \gp' \st A \supset A_1,\, A \cap
A_2 = \varnothing\}$.
Let $\sigma \colon \Gamma \to \bF_S$ be a section of the natural
surjection $\pi \colon \bF_S \to \Gamma$.  
For $v \in \vertex_n$, write
$$
A_{i,n,v} := v.\sigma(A_i \cap \gp) \cup \phi(v.\sigma(A_i \cap \phi(\gp)))
\subset \vertex_n \cup \vertex_n'
$$
for $i \in \{1,2\}$ and define
$$\mathcal A_{v,n}:=
\left\{ A \in 2^{\vertex_n} \times 2^{\vertex_n} \st A_{1,n,v} \subset A,
A_{2,n,v}  \cap A = \varnothing \right\} \subset 2^{\vertex_n} \times
2^{\vertex_n}\,.$$
We define
$$\tilde \mu(\mathcal A) := \lim_{n \to \omega} \frac1{|\vertex_n|}\sum_{v
\in \vertex_n}\mu_n(\mathcal A_{v,n})\,.$$
Since ultralimits are finitely additive, this extends to define
a consistent measure on cylinder events, whence by Kolmogorov's Extension
Theorem, 
there exists a unique measure $\mu$ on the Borel $\sigma$-algebra of
$2^{\gp} \times 2^{\gp}$ that extends $\tilde \mu$. We claim that
$\mu$ is the desired coupling.
It is a basic property of the formula
in the definition of $\tilde \mu$ that the action of $\bF_S$ on
$G_n$ just permutes the summands of the right-hand side. Hence, we conclude
that $\tilde \mu$ is $\Gamma$-invariant. Uniqueness of the extension in
Kolmogorov's Extension Theorem implies that $\mu$ is also
$\Gamma$-invariant. It follows from \rref l.ultradetl/
that the marginals of
$\mu$ are just $\P^{Q_1}$ and $\P^{Q_2}$. Finally, it is a monotone
coupling because each $\mu_n$ is monotone.
This finishes the proof.
\end{proof}

As a special case, we obtain the following:

\procl c.sfcouple
Let $\gp$ be a sofic group and $S$ be a finite generating set of elements
in $\gp$. 
There exists a $\Gamma$-invariant monotone coupling between $\wsf_G \dom
\fsf_G$ on the associated Cayley graph.
\endprocl

\rproof
It remains only to pass from the Cayley diagram to the Cayley graph.
Every edge of the Cayley graph is doubled in the diagram.
However, recall that in order to define the spanning forest measures, one
must first choose an orientation for each unoriented edge.
Thus, we may simply choose one of each pair in the diagram to be the
orientation of the corresponding unoriented edge. If we ignore the other
edge, then the corresponding determinantal probability measures are
precisely the ones we want when we identify each chosen oriented edge with
its corresponding unoriented edge. 
\Qed

The preceding corollary was proven by Bowen for residually amenable groups \cite{bowen}. Elek and Szab\'o \cite{elekszabo2} gave examples of finitely generated sofic groups that are not residually amenable. Later, Cornulier gave the first examples of finitely presented groups that are sofic but not residually amenable (or even limits of amenable groups) \cite[Corollary 3]{cornulier}. 

Similar reasoning shows, e.g., that if $0 \le Q_1 \le \cdots \le Q_r \le I$
in $R(\Gamma)$, then there exists a $\Gamma$-invariant sofic coupling of
all $\P^{Q_i}$ simultaneously that is monotone for each successive pair
$i$, $i+1$.

\bsection{Free uniform spanning forest measures as limits over a sofic
approximation}{s.approxim}

For certain determinantal measures we can say more. Let us first
introduce some more notation.
Note that for $d \in \bbN$, the embedding $(R(\Gamma),\tau) \subset
\prod_{n \to \omega} (B(\ell^2 \vertex_n),\tr_{\vertex_n})$ gives
rise to an embedding $$(M_dR(\Gamma),\tau^{(d)}) \subset \prod_{n \to \omega}
(M_dB(\ell^2 \vertex_n), \tr^{(d)}_{\vertex_n})\,,$$ where ${\rm
tr}^{(d)}_{\vertex_n}$ denotes the natural extension of the trace ${\rm
tr}_{\vertex_n}$ to $M_dB(\ell^2 \vertex_n)$. Similarly, we denote the natural
extension of $\tau$ to $M_d R(\Gamma)$ by $\tau^{(d)}$. 
Let $G=(\vertex,\edge)$ be an $S$-labelled Schreier graph. We shall
consider the natural extension of $\rho_{G}$ (defined after \eqref{defpi})
to $$\rho_G^{(d)} \colon M_d
\bbC \bF_S \to M_d B(\ell^2 \vertex)\,.$$

The following theorem is a variant of L\"uck's approximation theorem. In
the generality that we need, it was first proved by Elek and Szab\'o in
\cite[Proposition 6.1(a)]{elekszabo}. 
Let us state what we need in our notation.

\begin{thm}[Elek-Szab\'o] \label{eleksz}
Let $\gp$ be a sofic group and $S$ be a finite generating set of elements
in $\gp$. Let $(G_n)_n$ be a sequence of finite $S$-labelled Schreier
graphs whose limit is $G:={\rm Cay}(\gp,S)$. Let $d \in \bbN$ and $a \in
M_d \Z \bF_S$. Let $P_n \in M_d B(\ell^2 \vertex_n)$ denote the
projection onto the kernel of $\rho^{(d)}_{G_n}(a)$ and $P \in M_d R
\Gamma$ denote the projection onto the kernel of $\rho^{(d)}_G(a)$. Then
$$\lim_{n \to\infty} \tr^{(d)}_{\vertex_n}(P_n) = \tau^{(d)}(P)\,.$$
\end{thm}

The statement of the previous theorem is not just a consequence of weak*
convergence of spectral measures. The proof uses in an essential way the
integrality of coefficients of $a \in M_d \Z \bF_S$. The first
results of this form were obtained by L\"uck in \cite{lueck}, and over the
years they inspired many results of the same type.
Analogues of L\"uck's approximation theorem in the context of convergent sequences of finite graphs were studied in \cite{abertviragthom}.

Note that Proposition 6.1 in \cite{elekszabo}
contains a part (b), which asserted that 
\begin{equation} \label{claimb}
\int_{0^+}^{\infty} \log (t) \ d \mu_{|\rho^{(d)}_{G_n}(a)|}(t) \to \int_{0^+}^{\infty} \log (t) \ d \mu_{|\rho^{(d)}_{G}(a)|}(t) \quad \mbox{as } n \to \infty,\end{equation}
in a slightly more general form not assuming that the approximation is given by Schreier graphs; see \cite{elekszabo}. 
This part of the claim remained unproven in \cite{elekszabo}. Very
recently, it was discovered that \cite[Proposition 6.1(b)]{elekszabo} is
actually wrong (in the form more general than \eqref{claimb}).
It was shown independently by Lov\'asz and by
Grabowski-Thom that already the Cayley diagram of $G=\Z$ with respect to some
specific multi-set of generators admits a sofic approximation
$(G_n)_n$---albeit not by Schreier graphs---so that $|\vertex(G_n)|^{-1}
\log | \!\det A(G_n)|$ does not converge, where $A(G_n)$ denotes the
adjacency matrix of $G_n$. It is still possible that \eqref{claimb} holds
as written here.

\begin{cor} \label{soficapp}
Let\/ $\gp$ be a sofic group and $S$ be a finite generating set of elements
in $\gp$. Let $(G_n)_n$ be a sequence of finite $S$-labelled Schreier
graphs whose limit is $G:={\rm Cay}(\gp,S)$. Let $d \in \bbN$ and $a \in
M_d \Z \bF_S$. Let $P_n \in M_d B(\ell^2 \vertex_n)$ denote the
projection onto the kernel of $\rho^{(d)}_{G_n}(a)$ and $P \in M_d R
\Gamma$ denote the projection onto the kernel of $\rho^{(d)}_G(a)$. Then the
sequence $(P_n)_n$ represents $\iota_S(P) \in \prod_{n \to \omega}
(M_dB(\ell^2 \vertex_n),\tr^{(d)}_{\vertex_n})$.
\end{cor}
\begin{proof} Since $\ker(T^*T)=\ker(T)$ for any operator $T$ on a Hilbert
space, we may assume without loss of generality that $a = b^*b$ for some $b
\in M_d\bbC \bF_s$.
We set $T_n:= \rho^{(d)}_{G_n}(a)$ and $T:= \iota_S(\rho^{(d)}(a))$. It is
clear from Corollary \ref{emb} that the sequence $(T_n)_n$ represents $T$.
Set $c:= \sum_{i,j=1}^d \|a_{i,j}\|_1$,
where $a_{i,j}$ denotes the $(i,j)$-entry of the matrix $a$ and
$\|\sum_{\gamma} \alpha_{\gamma} \gamma\|_1 := \sum_{\gamma} |\alpha_{\gamma}|$. It
is easy to verify that $c \geq \sup \big\{ \|T_n\| \st n \in \bbN \big\}$
and $c \geq \|T\|$.
By Lemma \ref{lem:app}, for any continuous function $f \colon \bbR
\to \bbR$, the sequence $(f(T_n))_n$ represents $f(T)$. 
Now, let $(f_k)_k$ be a sequence of polynomials that are non-negative on
$[0, c]$ and satisfy 
$$\inf_k f_k(x) = \begin{cases} 1 & \mbox{if } x=0, \\ 0 & \mbox{if } x \in (0,c]. \end{cases}$$
For example, one can take $f_k(x) := (1-x/c)^{2k}$.
Note that $\inf_k f_k(T_n) = P_n$ and $\inf_k f_k(T)=\iota_S(P)$, where the infimum is taken with respect to the usual ordering on self-adjoint operators.
It is clear from Lemma \ref{lem:app} that $(P_n)_n$ represents a
projection that is smaller than $f_k(T)$ for each $k \in \bbN$ and hence
smaller than $\iota_S(P) = \inf_k f_k(T)$. 
L\"uck's approximation theorem (Theorem \ref{eleksz}) says that
$\lim_{n \to \infty} \tr^{(d)}_{\vertex_n}(P_n) = \tau^{(d)}(P) =
\tau^{(d)}_{\omega}(\iota_S(P))$. This shows that the subprojection of
$\iota_S(P)$ that $(P_n)_n$ represents has the same trace as
$\iota_S(P)$.
Since $\tau^{(d)}_{\omega}$ is faithful, this implies that $(P_n)_n$
represents $\iota_S(P)$. This finishes the proof.
\end{proof}

Let $a \in M_d \Z \bF_S$ be self-adjoint.  The heart of the proof
of the preceding theorem is that the convergence of spectral measures of
$T_n:=\rho^{(d)}_{G_n}(a)$ to the spectral measure of $T:=\rho^{(d)}(a)$
is far better than expected from Remark \ref{specmeas}, due to the
integrality of the coefficients of $a$.  Indeed, L\"uck's approximation
theorem asserts that $\lim_{n \to\infty} \mu_{T_n}(\{0\}) =
\mu_T(\{0\})$, which is not a consequence of weak convergence $\mu_{T_n}
\to \mu_T$ alone. In fact, even more is true. It is a consequence of
results in \cite{thom} that the integrated densities $F_{T_n}(\lambda):=
\mu_{T_n}((-\infty,\lambda])$ converge uniformly to $F_T(\lambda):=
\mu_T((-\infty,\lambda])$, i.e., $\sup\{ |F_{T_n}(\lambda) -
F_{T}(\lambda)| \st \lambda \in \bbR \} \to 0$ as $n \to \infty$. This
allows for more results like Theorem \ref{soficapp}, for example, for
other spectral projections of the operator $\rho^{(d)}(a)$.

Suppose that the random weak limit of $(G_n)$ is a Cayley graph $G$. In the
case where $G$ is amenable and $G_n$ are merely connected subgraphs of $G$,
we have that the random weak limit of the uniform spanning tree measures
$\ust_{G_n}$ equals $\wsf_G =
\fsf_G$. Despite the definition of $\fsf$, the random weak limit of
$\ust_{G_n}$ need not be $\fsf_G$
for a sofic approximation $(G_n)_n$ to a non-amenable group, $G$.
In fact, the limit of 
$\ust_{G_n}$ is always $\wsf_G$ (see \cite[Proposition
7.1]{AL:urn}). It is more complicated to get $\fsf_G$ as a limit.
The proof of \rref c.sfcouple/ provides such, but is not very explicit.
Here we give a more explicit method, which still provides an invariant
monotone coupling with $\wsf_G$.
For $L \ge 0$, let $\CYCLE_L(G)$ denote the space spanned by the cycles in
$G$ of length at most $L$.
Write $\fsf_{G, L}$ for the determinantal probability measure corresponding to
the projection onto $\CYCLE_L(G)^\perp$.
This measure is not necessarily concentrated on forests; rather, it is
concentrated on subgraphs of girth larger than $L$.
By \rref t.dominate/, we have $\ust_G \dom \fsf_{G, L}$ for all finite $G$ and
$L$.

\procl t.cycleslimit
If $G$ is a Cayley graph of a group $\gp$ and
is the random weak limit of $(G_n)_n$, then if $L(n) \to\infty$
sufficiently slowly,
the random weak limit of $\fsf_{G_n, L(n)}$ 
equals $\fsf_G$.
A subsequence of monotone couplings witnessing $\ust_{G_n} \dom \fsf_{G_n,
L(n)}$ has a $\gp$-invariant monotone coupling witnessing $\wsf_G \dom
\fsf_G$ as a weak* limit.
\endprocl

\rproof
By Corollary \ref{soficapp} and \rref l.ultradetl/, for all $L \ge 0$, the
random weak limit of $\fsf_{G_n, L}$ exists and equals $\fsf_{G, L}$. Since 
$\CYCLE_L(G) \uparrow \CYCLE(G)$,
the weak* limit of $\fsf_{G, L}$ equals $\fsf_G$.
A subsequence of monotone couplings witnessing $\ust_{G_n} \dom \fsf_{G_n,
L}$ has a $\gp$-invariant monotone coupling witnessing $\wsf_G \dom
\fsf_{G, L}$ as a random weak limit, which, as $L \to\infty$, has a
$\gp$-invariant monotone coupling witnessing $\wsf_G \dom \fsf_G$ as a
weak* limit.
The result follows. 
\Qed

In particular, with the assumptions of \rref t.cycleslimit/, calculation of
average expected degree shows that
$$
\lim_{L \to\infty} \lim_{n \to\infty} \frac{\dim
\CYCLE_L(G_n)}{|\vertex(G_n)|}
=
2\beta_1^{(2)}(\gp) + 2
\,.
$$
This equation is already known, as it follows from L\"uck's results and the
Determinant Conjecture, which was established for sofic groups in
\cite[Theorem 5]{elekszabo}.

\bsection{Consequences of the invariant monotone couplings}{s.conseq}

Given a network with positive edge
weights and a time $t > 0$, form the \dfnterm{transition operator} $P_t$ for
continuous-time random walk whose rates are the edge weights; in the case
of unbounded weights (or degrees), we take the minimal process, which dies
after an explosion.
That is, if the entries of a matrix $A$ indexed by the vertices are equal
off the diagonal to the negative of the edge weights and the diagonal
entries are chosen to make the row sums zero, then $P_t := e^{-A t}$; in
the case of unbounded weights, we take the self-adjoint extension of $A$
corresponding to the minimal process.
The matrix $A$ is called the \dfnterm{infinitesimal generator} or the
\dfnterm{Laplacian} of the network.

When the weights are random, we have a continuous-time random walk in a
random environment. If the distribution $\mu$ of the edge weights is
group-invariant, then $\E[P_t(o, o)] = \tr_\mu(e^{-A t})$.
Hence, if there are two sets of random weights, $A^{(1)}$ and $A^{(2)}$,
coupled by an invariant measure $\nu$ with the property that $A^{(1)}(e)
\le A^{(2)}(e)$ $\nu$-a.s.\ for all edges $e$, then the corresponding
return probabilities $P^{(1)}$ and $P^{(2)}$ satisfy $\E[P_t^{(1)}(o, o)]
\ge \E[P_t^{(2)}(o, o)]$: see \cite[Theorem 5.1]{AL:urn}.
Whether this inequality holds without assuming the existence of an
invariant coupling, but merely that the two weight distributions are
invariant, is open; it was asked by \rref b.FontesMathieu/. 

Of course, if the weights are simply the indicators of invariant random subsets,
then we obtain random walk on the random clusters. Thus, when we have an
invariant coupling of two percolation measures, we have the above
inequality on return probabilities. 
In particular, we have shown that such an inequality holds when the two
percolation measures are determinantal and arise from positive contractions
in $R(\gp)$.
A similar result holds when a more
complicated increasing function of the random subsets is used (such as
using for a weight of an edge the sum of the degrees of its
endpoints in the cluster), since given an invariant monotone coupling of
the two cluster measures, one easily constructs an invariant monotone
coupling of such weights.

For another consequence of our coupling result, we consider the $\FSF$.
It was proved in 
\BLPSusf\ that for every Cayley graph, whether sofic or not, a.s.\
each tree in $\wsf_G$ has one end. In addition, \BLPSusf\ also proved that
if $\fsf_G \ne \wsf_G$, then a.s.\ at least one tree in the $\FSF$ has
infinitely many ends. \BLPSusf\ conjectured that a.s.\ every tree in the
$\FSF$ has infinitely many ends in this case. We can now make a small
contribution to this conjecture:

\procl c.inftyends
If $G$ is the Cayley graph of a sofic group, then either $\fsf_G = \wsf_G$,
in which case a.s.\ each tree in $\fsf_G$ has one end, or $\fsf_G \ne \wsf_G$,
in which case a.s.\ each tree in $\fsf_G$ has one or infinitely many ends,
with some tree having infinitely many ends.
\endprocl

\rproof
Suppose that $\fsf_G \ne \wsf_G$.
Let $\mu$ be an invariant monotone coupling of the two spanning forest
measures.
Because $\WSF$ is spanning, each tree in $\FSF$ consists of unions of
(infinite) trees in $\WSF$ with additional connecting edges.
If there are only finitely many connecting edges, say, $N$, in some tree
$T$ in $\FSF$, then each vertex in $T$ can send mass $1/N$ to each
endpoint of each of the connecting edges in $T$. Such endpoints would
receive infinite mass, yet no point would send out mass more than 2.
Thus, the Mass-Transport Principle tells us that this event has probability
0.
Therefore, there are a.s.\ either no connecting edges or infinitely many in
each tree of $\FSF$. Combining this with what was previously known
gives the corollary. 
\Qed

The above consequences of \rref t.gencouple/ were for specific models. We
close with a general consequence that is relevant in ergodic theory.

For a set $X$, write $\pi_x \colon \{0, 1\}^X \to \{0, 1\}$ for the
natural coordinate projections ($x \in X$).
For $K \subseteq X$, write $\fd(K)$ for the $\sigma$-field on
$\{0, 1\}^X$ generated the maps $\pi_x$ for $x \in K$. 
When $X$ is the vertex set $\verts$ of a graph,
a probability measure $\mu$ on $\{0, 1\}^\verts$ is called
\dfnterm{$m$-dependent} if $\fd(K_1), \dots, \fd(K_p)$ are independent whenever
the sets $K_i$ are pairwise separated by graph distance $> m$. A similar
definition holds when $X$ is the edge set of a graph.
We say that
$\mu$ is \dfnterm{finitely dependent} if it is $m$-dependent for some $m <
\infty$.

Note that if $Q \in \bbC\Gamma$ is a positive contraction, then $\P^Q$ is
finitely dependent.
The Kaplansky density theorem implies that
every positive contraction $Q \in R(\Gamma)$ is the limit in the
strong operator topology (SOT) of positive contractions $Q_n \in \bbC\Gamma$ (see
\cite[Cor.~5.3.6]{KR1}). Combining these two observations, we see
that $\P^Q$ is the weak* limit of
the finitely dependent measures $\P^{Q_n}$.
Likewise, if $Q \in R(\gp, S)$ is a positive contraction, then there are
positive contractions $Q_n \in M_S(\bbC\gp)$ such that 
$\P^Q$ is the weak* limit of the finitely dependent measures $\P^{Q_n}$.

In the case that $\Gamma$ is sofic, we can strengthen weak* convergence to
$\dbar$-convergence because of our
monotone coupling result. This follows ideas of \rref b.LS:dyn/, but that
case, where $\gp$ is commutative, is much easier.

Let $\mu_1$ and $\mu_2$ be two $\gp$-invariant probability measures on $A^W$,
where $\gp$ acts quasi-transitively on $W$ and $A$ is finite. Let
$W'$ be a section of $\gp\backslash W$. Then 
Ornstein's $\dbar$-metric is defined as
$$
\dbar(\mu_1, \mu_2)
:=
\min \Big\{\sum_{w \in W'} \Pbig{X_1(w) \ne X_2(w)} \st X_1 \sim \mu_1,\, X_2
\sim \mu_2,\, (X_1, X_2) \hbox{ is $\gp$-invariant} \Big\}
\,.
$$
This is a metric for the following reason.
Suppose that $(X_1, X_2)$ is a joining of $(\mu_1, \mu_2)$ and $(X_3, X_4)$
is a joining of $(\mu_2, \mu_3)$. Given a Borel set $C \subseteq A^W$,
write $f_C(X_2) := \P[X_1 \in C \mid X_2]$ and
$g_C(X_3) := \P[X_4 \in C \mid X_3]$. The \dfnterm{relatively independent
joining of $(X_1, X_2)$ and $(X_3, X_4)$ over $\mu_2$} is defined
to be the measure $\mu$ on $(A^W)^3$ determined by 
$$
(C_1, C_2, C_3)
\mapsto
\int_{C_2} f_{C_1}(y) g_{C_3}(y) \,d\mu_2(y)
$$
for $C_1, C_2, C_3 \subseteq 2^W$.
It is easily verified and well known that this measure $\mu$ is indeed
$\gp$-invariant, and therefore a joining.
(Intuitively, we merely choose $X_2 = X_3$ to create this joining out of
the original pair of joinings. More precisely, $X_1$ and $X_4$ are then
chosen independently given $X_2$.)
Now choose the joinings $(X_1, X_2)$ and $(X_3, X_4)$ to achieve the minima
in the definition of $\dbar$.
If $(Y_1, Y_2, Y_3) \sim \mu$, then $\dbar(\mu_1, \mu_3) \le \sum_{w \in
W'} \Pbig{Y_1(w) \ne Y_3(w)} \le \sum_{w \in
W'} \Pbig{Y_1(w) \ne Y_2(w)} + \sum_{w \in
W'} \Pbig{Y_2(w) \ne Y_3(w)} =
\sum_{w \in
W'} \Pbig{X_1(w) \ne X_2(w)} + \sum_{w \in
W'} \Pbig{X_2(w) \ne X_3(w)} = \dbar(\mu_1, \mu_2) + \dbar(\mu_2, \mu_3)$,
as desired.

If $(\mu_1, \mu_2, \ldots, \mu_n)$ is a sequence of $\gp$-invariant
probability measures on $A^W$ and $\mu_{k, k+1}$ is a joining of $(\mu_k,
\mu_{k+1})$ for each $k = 1, 2, \ldots, n-1$, then there is an associated
relatively independent joining of all $n$ measures obtained by successively
taking a relatively independent joining $(Y_1, Y_2, Y_3)$ of $(\mu_1, \mu_2)$
with $(\mu_2, \mu_3)$ over $\mu_2$, then the relatively independent joining
of $(Y_1, Y_2, Y_3)$ with $(\mu_3, \mu_4)$ over $\mu_3$, where we regard
$(Y_1, Y_2, Y_3) \in (A \times A)^W \times A^W$, etc.
By taking a limit of such joinings, we can do the same for an infinite
sequence of invariant measures on $A^W$ with given successive joinings.

In case $A = \{0, 1\}$ and there is a monotone joining of $\mu_1$ and $\mu_2$,
then such a joining may be used to calculate $\dbar(\mu_1, \mu_2)$:

\procl l.dbarmonocalc
Let $\mu_1$ and $\mu_2$ be two $\gp$-invariant probability measures on $2^W$,
where $\gp$ acts quasi-transitively on $W$. Let
$W'$ be a section of\/ $\gp\backslash W$. If there is a monotone joining
$(X_1, X_2)$ of
$(\mu_1, \mu_2)$, then 
$$
\dbar(\mu_1, \mu_2)
=
\sum_{w \in W'} |\P[X_1(w) = 0] - \P[X_2(w) = 0]|
=
\sum_{w \in W'} \Pbig{X_1(w) \ne X_2(w)}
\,.
$$
Suppose that in addition,
$\mu_3$ and $\mu_4$ are two $\gp$-invariant probability measures on $2^W$
such that there are monotone joinings witnessing
$\mu_1 \dom \mu_3 \dom \mu_2$ and
$\mu_1 \dom \mu_4 \dom \mu_2$.
Then $\dbar(\mu_3, \mu_4) \le \dbar(\mu_1, \mu_2)$.
\endprocl

\rproof
It is clear that any joining $(X_1, X_2)$ of $(\mu_1, \mu_2)$ has the
property that 
$$
\sum_{w \in W'} \Pbig{X_1(w) \ne X_2(w)} \ge 
\sum_{w \in W'} |\P[X_1(w) = 0] - \P[X_2(w) = 0]|
$$
and that a monotone joining gives equality. 
Furthermore, if we extend $(X_1, X_2)$ to a relatively independent joining
$(X_1, X_2, X_3, X_4)$ with the assumed joinings 
satisfying $X_1 \le X_3 \le X_2$ and $X_1 \le X_4 \le X_2$, then the
joining $(X_3, X_4)$ witnesses the desired inequality.
\Qed

We shall prove the following:

\procl t.dbarfindep
Let $\gp$ be a sofic group and $S$ be a finite generating set of elements
in $\gp$. 
If $Q$ is a positive contraction in $R(\gp)$ or in $R(\gp, S)$, then there
exists a sequence of positive contractions $Q_n$ in $\bbC\gp$ or
in $M_S(\bbC\gp)$ such that the finitely dependent probability measures
$\P^{Q_n}$ converge to $\P^Q$ in the $\dbar$-metric.
\endprocl

Note that when $\gp$ is amenable, \rref t.gencouple/ is easy.
In addition, when $\gp$ is amenable, it is known that 
the $\gp$-invariant finitely dependent processes are isomorphic to
Bernoulli shifts by using the very weak Bernoulli condition of \rref
b.Orn:book/, extended to the amenable case by \rref b.Adams/; 
that
factors of Bernoulli shifts are isomorphic to Bernoulli shifts
\cite{Orn:factor,OrnW:amen};
and that
the class of processes isomorphic to Bernoulli shifts is $\dbar$-closed
\cite{Orn:factor,OrnW:amen}.
Thus, we have the following corollary, which was proved for abelian $\gp$
in \cite{LS:dyn}:

\procl c.dbarFIID
Let $\gp$ be an amenable group and $S$ be a finite generating set of elements
in $\gp$. 
If $Q$ is a positive contraction in $R(\gp)$ or in $R(\gp, S)$, then 
$\P^Q$ is isomorphic to a Bernoulli shift.
\endprocl

On the other hand, in the non-amenable setting,
Popa gave an example of a factor of a Bernoulli shift that is not
isomorphic to a Bernoulli shift. Indeed, \cite[Corollary 2.14]{Popa} showed
that for any infinite group $\Gamma$ 
with Kazhdan's Property (T), the natural action $\Gamma
\curvearrowright ({\mathbb T},\mu_{\rm Haar})^{\Gamma}/(\Z/n\Z)$ is not
isomorphic to a Bernoulli shift for any $n \geq 2$. Here, $\Z/n\Z$ is
understood to act diagonally on $({\mathbb T},\mu_{\rm Haar})^{\Gamma}$ by
rotation in the obvious way.  For such $\gp$, it follows that the natural
action $\Gamma \curvearrowright (\Z/n\Z,\mu_{\rm Haar})^{\Gamma}/(\Z/n\Z)$
is not isomorphic to a Bernoulli shift for any $n \geq 2$.

Natural questions, therefore, include these, which are all settled in the
amenable case:

\procl q.findep
Is every finitely dependent process a factor of a Bernoulli shift?
\endprocl

\procl q.dbar
Let $\gp$ act quasi-transitively on a countable set $W$ and let $A$ be finite.
Is the class of measures on $A^W$ that are factors of Bernoulli shifts
closed in the $\dbar$-metric?
\endprocl

\procl q.detlFIID
Are determinantal probability measures associated to equivariant positive
contractions factors of Bernoulli shifts? 
\endprocl

By \rref t.dbarfindep/, positive answers to Questions \ref{question:findep}
and \ref{question:dbar} would imply a positive answer to \rref
q.detlFIID/ on sofic groups.

In order to prove \rref t.dbarfindep/, we shall use two lemmas.
We give statements and
details for $R(\gp)$; they
admit straightforward extensions to $R(\Gamma,S)=M_S(R(\gp))$.

\procl l.dbarnorm
If $Q$ and $Q'$ are positive contractions in $R(\gp)$, then
$$
\dbar\big(\P^Q, \P^{Q'}\big) \le \frac{6 \|Q - Q'\|}{1 + 2\|Q - Q'\|}
\,.
$$
\endprocl

\rproof
Write $r := \|Q - Q'\|$ and $t := r/(1+2r)$. We set $Q_t := (1-t)Q+ t(I-Q)$.
Then 
$
Q \ge (1-t)Q$ and $Q_t \ge (1-t)Q$,
whence by the triangle inequality and \rref l.dbarmonocalc/, we have
$$
\dbar\big(\P^Q, \P^{Q_t}\big)
\le
\dbar\big(\P^Q, \P^{(1-t)Q}\big)
+
\dbar\big(\P^{(1-t)Q}, \P^{Q_t}\big)
\le
\|tQ\| + \|t (I-Q)\|
\le 2t
\,.
$$
Likewise, with $Q'_t := (1-t)Q'+ t(I-Q')$, we have 
$$
\dbar\big(\P^{Q'}, \P^{Q'_t}\big)
\le 2t
\,.
$$
In addition, $tI \le Q'_t \le (1-t)I$ and $Q_t - Q'_t = (1-2t)(Q-Q')$ has
norm $(1-2t)r = t$, whence
$$
0 \le Q'_t - tI \le Q_t \le Q'_t + tI \le I
\,.
$$
\rref l.dbarmonocalc/ again yields
$$
\dbar\big(\P^{Q_t}, \P^{Q'_t}\big) 
\le
\dbar\big(\P^{Q'_t-tI}, \P^{Q'_t+tI}\big) 
=
2t
\,.
$$
Putting together these inequalities and using the triangle inequality for
the $\dbar$-metric gives $\dbar\big(\P^Q, \P^{Q'}\big) \le 6t$, which is
the desired result. 
\Qed

When $Q$ and $Q'$ commute,
one can improve the bound in \rref l.dbarnorm/ by replacing the norm on the
right-hand side with the Schatten 1-norm. Recall that this norm is 
$$
\|T\|_1 := \tau\big((T^* T)^{1/2}\big)
\,.
$$
In this language,
when $\Gamma$ is abelian, \rref b.LS:dyn/ showed that
$\dbar(\P^Q, \P^{Q'}) \le \|Q - Q'\|_1$.
In fact, the same proof can be adapted for all $\gp$ to the case that
$Q$ and $Q'$ commute. We do not know
whether this inequality always holds, but we have the following weaker
version:

\procl l.dbarSchatten
If $Q$ and $Q'$ are positive contractions in $R(\gp)$, then
$$
\dbar\big(\P^Q, \P^{Q'}\big) \le 
6 \cdot 3^{2/3} \|Q - Q'\|_1^{1/3}
\,.
$$
If $Q_n$ and $Q$ are positive contractions in $R(\gp)$ with
$Q_n \to Q$ in {\rm SOT}, then 
$\dbar\big(\P^{Q_n}, \P^Q\big) \to 0$.
\endprocl

\rproof
We shall use the Schatten 2-norm, $\|T\|_2 := \sqrt{\tau (T^* T)}$, and
the Powers-St{\o}rmer inequality, $\|T_1 - T_2\|^2_2 \le \|T_1^2 -
T_2^2\|_1$ for $0 \le T_1, T_2 \in R(\gp)$; see \cite[Proposition
6.2.4]{BroOz} for a proof that extends to our context.

Write $T := Q^{1/2}$ and $T' := (Q')^{1/2}$.
Let $E$ be the spectral resolution of the identity for $T-T'$, so that
$$
T-T' = \int_{-1}^1 s\,dE(s)
\,.
$$
Thus, 
$$
r 
:= 
\int_{-1}^1 s^2 \,d\nu(s)
=
\|T-T'\|_2^2
\le
\|Q - Q'\|_1
$$
for the scalar measure $A \mapsto \nu(A) := \tau \big(E(A)\big)$.
Define $t := (r/3)^{1/3}$ and $B := [-1, -t] \cup [t, 1]$.
We have $\nu(B) \le r/t^2$ by Markov's inequality.
Write $P := E(B)$ and $P^\perp := E\big((-t, t)\big)$.
Define $Q_1 := T P T$
and $Q'_1 := T' P T'$, and write $Q_2 := Q-Q_1$, $Q'_2 := Q'-Q'_1$.
These are all positive contractions in $R(\gp)$; for example, $Q_2 = T
P^\perp T = (P^\perp T)^* (P^\perp T) \ge 0$.
Furthermore, 
$$
Q_2 - Q'_2
=
T P^\perp (T - T') + (T-T') P^\perp T'
$$
and $\|T\|, \|T'\| \le 1$,
whence $\|Q_2 - Q_2'\| \le 2 \|(T - T') P^\perp\|$. Since
$$
(T - T') P^\perp
=
\int_{(-t, t)} s \,dE(s)
\,,
$$
it follows that $\|(T - T') P^\perp\| \le t$, whence
$\|Q_2 - Q'_2\| \le 2t$ and $\dbar\big(\P^{Q_2},
\P^{Q'_2}\big) \le 12t$ by \rref l.dbarnorm/.
Now, $\tau (Q_1) = \tau (T^2 P) = \tau (T^2 P^2) = \tau (P T^2 P) = \|T P
\delta_o\|^2 \le \|T\|^2 \|P \delta_o\|^2 \le \tau (P) = \nu(B) \le r/t^2$.
Since $Q_2 \le Q$, it follows that $\dbar\big(\P^Q, \P^{Q_2}\big) =
\tau(Q-Q_2) = \tau (Q_1) \le r/t^2$.
Likewise, $\dbar\big(\P^{Q'}, \P^{Q'_2}\big) \le r/t^2$.
Therefore, 
$$
\dbar\big(\P^Q, \P^{Q'}\big)
\le
\dbar\big(\P^Q, \P^{Q_2}\big) + \dbar\big(\P^{Q_2}, \P^{Q'_2}\big) +
\dbar\big(\P^{Q'_2}, \P^{Q'}\big)
\le
\frac{r}{t^2} + 12t + \frac{r}{t^2}
=
6 (9r)^{1/3}
\,,
$$
as desired. 

The second part of the assertion follows from the inequality
$$\|Q_n \delta_o - Q
\delta_o\| = \|Q_n - Q\|_2 \ge \|Q_n - Q\|_1\,.$$
This finishes the proof.\Qed

We remark that
if it is assumed only that $Q_n$ converges to $Q$ in WOT, then it does {\em
not} follow that $\P^{Q_n}$ converges to $\P^Q$ in $\dbar$. In fact, on
$\Z$, entropies need not converge: see the end of Sec.~6 of \rref b.LS:dyn/
for examples.
Here, we are using the fact that for processes on $\Z$, entropy is
$\dbar$-continuous \cite[Proposition 15.20]{Glasner}.

\begin{proof}[Proof of \rref t.dbarfindep/:]
This is immediate from \rref l.dbarSchatten/ and Kaplansky's Density Theorem, i.e., the fact that positive contractions in 
$R(\gp)$ are SOT-limits of positive contractions in $\C\gp$.
\end{proof}

By analogy with Bernoulli processes in the amenable case,
one could ask whether determinantal processes are finitely determined,
where we could define $\mu$ as \dfnterm{finitely determined} if whenever
$\mu_n \to \mu$
weak* and in sofic entropy, we have $\dbar(\mu_n, \mu) \to 0$. The converse
presumably holds for all processes. When $\gp$ is amenable, this is
known since sofic entropy equals ordinary metric entropy.
Here, we are relying on definitions and results of \cite{Bowen:sofic}.

Numerical calculation suggests that the inequality 
$\dbar(\P^Q, \P^{Q'}) \le \|Q - Q'\|_1$ always holds, even for finite
matrices without any invariance. Our proof of the weaker inequality \rref
l.dbarSchatten/ holds in that generality.
These inequalities appear to imply similar inequalities for continuous
determinantal point processes $(\sX, \sY)$, where the $\dbar$-metric is
replaced by taking the minimum over all joinings of the intensity of the
symmetric difference $\sX \triangle \sY$; we plan to pursue this elsewhere.

\bsection{Unimodular random rooted graphs}{s.unimodular}

We now extend the preceding theorems to their natural setting encompassing
all random weak limits of finite graphs with bounded degree (and somewhat
beyond).
One other setting in which it would be natural to investigate these
questions is that of vertex-transitive graphs and their
automorphism-invariant determinantal probability measures.
However, we are able to treat only the sofic ones (again), which, in
particular, excludes all non-unimodular transitive graphs.

We review a few definitions from the theory of unimodular random rooted
networks; for more details, see \rref b.AL:urn/.
A \dfnterm{network} is a (multi-)graph $\gh = (\vertex, \edge)$ together with a
complete separable metric space $\marks$ called the \dfnterm{mark space} and
maps from $\vertex$ and $\edge$ to $\marks$. Images in $\marks$ are called
\dfnterm{marks}.
The only assumption on degrees is that they are finite when loops are not
counted.
We omit the mark maps from our notation for
networks. 

A \dfnterm{rooted network} $(\gh, \bp)$ is a network $\gh$ with a distinguished
vertex $\bp$ of $\gh$, called the \dfnterm{root}.
A \dfnterm{rooted isomorphism} of rooted networks is an isomorphism of the
underlying networks that takes the root of one to the root of the other.
We do not distinguish between a rooted network and its
isomorphism class.
Let $\GG_*$ denote the set of rooted isomorphism classes of rooted
{\it connected\/} locally finite networks.
Define a separable complete metric $d_* \colon \GG_* \times \GG_* \to [0,1]$ on $\GG_*$ by letting the distance
between $(G_1, o_1)$ and $(G_2, o_2)$ be $1/(1+\alpha)$, where $\alpha$ is
the supremum of those $r > 0$ such that there is some rooted isomorphism of
the balls of (graph-distance) radius $\flr{r}$ around the roots of $G_i$
such that each pair of corresponding marks has distance less than $1/r$.
For probability measures $\rtd$, $\rtd_n$ on $\GG_*$, we write $\rtd_n \cd
\rtd$ when $\rtd_n$ converges weakly with respect to this metric.

For a (possibly disconnected)
network $\gh$ and a vertex $x \in \verts(\gh)$, write $\gh_x$ for the
connected component of $x$ in $\gh$.
If $\gh$ is finite, then write $U_\gh$ for a uniform random vertex of $\gh$
and $U(\gh)$ for the
corresponding distribution of $\big(\gh_{U_\gh}, U_\gh\big)$ on $\GG_*$.
Suppose that $\gh_n$ are finite networks and that $\rtd$ is a
probability measure on $\GG_*$.
We say that the \dfnterm{random weak limit} of a sequence $(\gh_n)_n$ is $\rtd$ if $U(\gh_n)
\cd \rtd$. 

A probability measure
that is a random weak limit of finite networks is called \dfnterm{sofic}.
In particular, a group is called sofic when its Cayley diagram is sofic.

All sofic measures are unimodular, which we now define.
Similarly to the space $\GG_*$, we define the space $\gtwo$ of isomorphism
classes of locally
finite connected networks with an ordered pair of distinguished vertices
and the natural topology thereon.
We shall write a function $f$ on $\gtwo$ as $f(\gh, x, y)$.
We refer to $f(\gh, x, y)$ as the \dfnterm{mass} sent from $x$ to $y$ in $\gh$.

\procl d.unimodular 
Let $\rtd$ be a probability measure on $\GG_*$.
We call $\rtd$ \dfnterm{unimodular} if it obeys the \dfnterm{Mass-Transport
Principle}:
For all Borel
$f \colon \gtwo \to [0, \infty]$,
we have
$$
\int \sum_{x \in \vertex(\gh)} f(\gh, \bp, x) \,d\rtd(\gh, \bp)
=
\int \sum_{x \in \vertex(\gh)} f(\gh, x, \bp) \,d\rtd(\gh, \bp)
\,.
$$
\endprocl

It is easy to see that every sofic measure 
is unimodular, as observed by \rref b.BS:rdl/, who introduced this
general form of the Mass-Transport Principle under the name 
``intrinsic Mass-Transport Principle".
The converse was posed as a question by \rref b.AL:urn/; it remains open.

Consider the Hilbert space $\Hilb(\mu) := \int^\oplus \ell^2\big(\vertex(\gh)\big)
\,d\rtd(\gh, \bp)$, a direct integral (see, e.g., \rref b.Nielsen/ or
\cite[Chapter 14]{KR}).
Here, we always choose canonical representatives for rooted-isomorphism
classes of networks, as explained
in \cite[Sec.~2]{AL:urn}; in particular, $\verts(\gh) = \bbN$. However,
this is merely for technical reasons of measurability, so we omit this from
our notation.
The space $\Hilb(\mu)$ is defined as the set
of ($\rtd$-equivalence classes of) $\rtd$-measurable functions $\xi$ defined on
(canonical) 
rooted networks $(\gh, \bp)$ that satisfy $\xi(\gh, \bp) \in
\ell^2(\verts(\gh))$ and $\int \|\xi(\gh, \bp)\|^2 \,d\rtd(\gh, \bp) < \infty$.
We write $\xi = \int^\oplus \xi(\gh, \bp) \,d\rtd(\gh, \bp)$.
The inner product is given by $\iprod{\xi, \eta} := \int \iprod{ \xi(\gh, \bp), \eta(\gh,
\bp } \,d\rtd(\gh, \bp)$.
Let $T \colon (\gh, \bp) \mapsto T_{\gh, \bp}$ be a measurable assignment of
bounded linear operators on $\ell^2\big(\vertex(\gh)\big)$
with $\mu$-finite
supremum of the norms $\|T_{\gh, \bp}\|$.
Then $T$ induces a bounded linear operator $T := T^\rtd := \int^\oplus T_{\gh,
\bp} \,d\rtd(\gh, \bp)$ on $\Hilb$ via
$$
T^\rtd \colon \int^\oplus \xi(\gh, \bp)
\,d\rtd(\gh, \bp) \mapsto \int^\oplus T_{\gh, \bp} \xi(\gh, \bp)
\,d\rtd(\gh, \bp)
\,.
$$
The norm $\|T^\rtd\|$ of $T^\rtd$ is the $\rtd$-essential supremum of
$\|T_{\gh, \bp}\|$.
We say that $T$ as above is \dfnterm{equivariant} if for all network isomorphisms $\phi \colon \gh_1 \to \gh_2$ preserving the marks, all
$\bp_1, x, y \in \verts(\gh_1)$ and all $\bp_2 \in \verts(\gh_2)$,
we have $\iprod{T_{\gh_1, \bp_1} \delta_x, \delta_y} = \iprod{T_{\gh_2,
\bp_2} \delta_{\phi(x)}, \delta_{\phi(y)}}$. For $T \in B(\Hilb(\mu))$ equivariant, we have in particular that $T_{\gh, \bp}$ depends on
$\gh$ but not on the root $\bp$, so we shall simplify our notation and
write $T_\gh$ in place of $T_{\gh, \bp}$.
For simplicity, we shall even write $T$ for $T_\gh$ when no confusion can arise.

We shall show that sofic probability measures can be extended to sofic
measures on $S$-labelled Schreier networks; any new loops will get new marks
indicating that they were not in the original underlying graph.
This will be an important technical tool. It can only enlarge the class of
equivariant operators.

\procl p.addS
Let $(G_n)_n$ be networks with finitely many vertices and edges
whose random weak limit is $\mu$ and whose
mark space is $\marks$.
Let $|S|$ be at least twice the degree of every vertex in every
$G_n$; possibly $|S| = \infty$.
Then there exist $S$-labelled Schreier
networks $H_n$ with mark space $\marks \times \{0, 1\}$
with the following properties:
\begin{enumerate}
\item[(i)] 
The underlying graph of each component of
$H_n$ is equal to that of\/ $G_n$ except
that $H_n$ may have additional loops whose second mark coordinate is 1.
\item[(ii)] 
The first coordinate marks of each component of
$H_n$ restricted to the underlying graph of $G_n$
agree with the marks on $G_n$.
\item[(iii)] 
The sequence $(H_n)_n$ has a random weak limit carried by $S$-labelled
Schreier networks.
\end{enumerate}
\endprocl

\rproof
Given a locally finite
network $G$ with mark space $\marks$, produce a random $S$-labelled
Schreier network, $\phi(G)$, with mark space $\marks \times \{0, 1\}$ as
follows.
Let $\mk_0$ be an arbitrary element of $\marks$.
Let $U_k(e)$ be independent uniform $[0, 1]$ random variables for $k \ge 1$
and $e \in \edges(G)$.
For an edge $e$, let $N(e)$ be the set of edges (including $e$) that share
an endpoint with $e$.
Write $S = \{s_1, s_2, \ldots\}$.
We shall use the identity map for the involution $i \colon S \to S$.
Assign a second mark coordinate of 0 to every vertex and edge of $G$.
Assign the label $s_1$ to every edge $e$ such that $U_1(e) = \min \{U_1(e')
\st e' \in N(e)\}$.
We assign further labels from $S$ recursively.
Supposing that a partial assignment has been made using the random
fields $U_1, \ldots, U_k$, let $J(e)$ be the minimum index $j$ such that 
no edge in $N(e)$ has been assigned the label $s_j$, if any.
By choice of $|S|$, there is always such an index when $e$ does not yet
have a label.
Now assign the label $s_{J(e)}$ to every $e$ that does not have a label and
for which $U_{k+1}(e) = \min \{U_{k+1}(e') \st e' \in N(e)\}$.
After all these assignments are completed, to every vertex $x$,
add new loops with mark $(\mk_0, 1)$ so that the degree of $x$ is equal
to $|S|$ and so that the resulting network, $\phi(G)$, is $S$-labelled.

Except for the fact that $\phi(G_n)$ is random, the sequence
$\big(\phi(G_n)\big)_n$ has all the desired properties: note that
$U\big(\phi(G_n)\big) \Rightarrow \nu$, where for a measurable set $A$ of
rooted networks,
$$
\nu(A)
:=
\int \P\big[\big(\phi(G), o\big) \in A\big] \,d\mu(G, o)
\,.
$$

To fix this problem, let $\mk_1, \mk_2, \ldots$ be a dense subset of
$\marks$.
Let $\psi_n \colon \marks \to \{\mk_1, \ldots, \mk_n\}$ be a map that takes
each mark $\mk$ to one of the closest points to it among $\{\mk_1, \ldots,
\mk_n\}$.
Then $\psi_n$ naturally induces a map $\tilde \psi_n$ on networks.
The push-forward $\nu_n$ of the law of
$U\big(\phi(G_n)\big)$ by $\tilde \psi_n$ gives a finitely supported
probability measure.
By taking a rational approximation of its probabilities, we may find a
finite (disconnected)
network $H_n$ such that $U(H_n)$ is within total variation distance
$1/n$ of $\nu_n$.
This is the sequence desired. 
\Qed

Let $(X,d)$
be a separable compact metric space. From now on, we shall assume that
$G=(\vertex,\edge)$ is a rooted connected $S$-labelled Schreier network.
Moreover, we assume that the
{\em vertex} labels take values in the space $X$. If $S$ is finite, then
the space of marks $\marks = S \cup X$ is compact. We denote the set of
such (rooted connected $S$-labelled Schreier) networks by $\sGG$ or
$\sGG(S,X)$.
We use the following metric on $\sGG(S, X)$:
Write $S = \{s_1, s_2, \ldots\}$. 
For a rooted $S$-labelled Schreier network $(G, o)$ and $n \ge 1$, let
$G^{n}$ denote the connected component of $o$ in the subnetwork of $G$ formed
by deleting all edges with a label $s_k$ (in either direction) for any $k > n$.
Define a separable complete metric $d_* \colon \sGG \times \sGG \to [0,1]$
on $\sGG$ by letting the distance
between $(G_1, o_1)$ and $(G_2, o_2)$ be $1/(1+\alpha)$, where $\alpha$ is
the supremum of those $r > 0$ such that there is some rooted isomorphism of
the balls of (graph-distance) radius $\flr{r}$ around the roots of
$G_i^{\flr{r}}$ that preserves marks up to an error of at most $1/r$ in the
metric of the mark space.
Even if $S$ is infinite,
$\sGG(S,X)$ is a compact
metric space (basically because $\{0, 1\}^S$ is compact)
and thus for any sequence $(\mu_n)_n$ of probability
measures on $\sGG(S,X)$, we have $\mu_n \Rightarrow \mu$ if and only if
$\mu_n \to \mu$ in the weak* topology. 
For simplicity of notation, we omit all other edge marks; one may actually
encode edge marks via vertex marks in any case.

Before we proceed, let us discuss a natural example.  Let $\Gamma$ be a
group  that acts by
homeomorphisms on a compact metrizable space $X$, preserving a probability
measure, $\mu$. We assume that $\Gamma$ is generated by a finite symmetric set $S \subset \Gamma$ and define the involution $i \colon S \to S$ by $i(s)=s^{-1}$. Then we can associate to each point $x \in X$ the
$S$-labelled rooted Schreier graph that arises from the restriction of the
action of $\Gamma$ to the orbit of $x$. Each vertex in this graph carries a
natural label in $X$. We obtain a continuous map $\varphi \colon X \to
\sGG(S,X)$ and can consider the push-forward measure $\varphi_*(\mu)$.
This measure is unimodular. Moreover, there is a natural equivariant factor
map $\psi \colon \sGG(S,X) \to X$, sending a rooted graph to the label of
its root.

Suppose that $\rtd$ is a unimodular probability measure on $\sGG(S,X)$.
Note that there is a continuous action of $\bF_S$ on the compact space
$\sGG(S,X)$ that moves the root according to the labels seen at the root.
Moreover, this action preserves the measure $\mu$, since $\mu$ is
unimodular. Consider the ring $C(\sGG(S,X))$ of continuous functions on
$\sGG(S,X)$ and the \dfnterm{algebraic crossed product algebra} $C(\sGG(S,X))
\rtimes \bF_S$. Recall that
$C(\sGG(S,X)) \rtimes \bF_S$ is a $*$-algebra and consists of finite
formal sums $\sum_{w \in \bF_S} f_w w$ with $f_w \in C(\sGG(S,X))$. The
multiplication and involution are defined by linearity and the formulas
\begin{align*}
\forall f_1,f_2 \in C(\sGG(S,X)) \ \ \forall w_1,w_2 \in \bF_S \quad (f_1
w_1)
\cdot (f_2 w_2) &:= f_1 (\prescript{w_1}{}{\!f_2})  w_1w_2\,,\\
(f_1 w_1)^* &:=
(\prescript{w_1^{-1}}{}{\!\bar f_1})
w_1^{-1}\,,
\end{align*}
where we use the convention $\prescript{w}{}{\!f}(G,v) := f(G,v.w)$. The measure $\mu$ gives rise to a functional $\tau_{\mu}$ on the crossed product algebra as follows:
$$\tau_{\mu}\left(\sum_{w \in \bF_S} f_w w \right) := \sum_{w \in \bF_S}
\int \bfone_{o.w = o} \cdot f_{w}(G,o) \
d\mu(G,o) \,.$$

There are two natural actions, $F$ and $M$, on $\Hilb(\mu)$
of the algebra of continuous functions on
$\sGG(S,X)$, which will both be of importance. First of all, $f \in
C(\sGG(S,X))$ can act as a constant on the fibers, i.e., $F(f)_{G,o} :=
f(G,o) \cdot I_{\ell^2(\verts(G))}$, or equivalently
$$
F(f)(\xi)
:=
\int^\oplus f(\gh, \bp) \xi(\gh, \bp) \,d\rtd(\gh, \bp)
\,.
$$ 
It is a basic fact that an operator $T
\in B(\Hilb(\mu))$ arises as above
from a measurable family $(G,o) \mapsto T_{G,o}$
iff $T$ commutes with $F(C(\sGG(S,X))$.
A second action of $C(\sGG(S,X))$ is defined by the formula $M(f)_{G,o}
\delta_v = f(G,v) \cdot \delta_v$ for all $v \in \verts(G)$, in other
words, 
$$
M(f)(\xi)
:=
\int^\oplus M(f)_{\gh, \bp} \xi(\gh, \bp) \,d\rtd(\gh, \bp)
\,,
$$
where $M(f)_{G,o} \xi(G, o)(v) := f(G,v) \cdot \xi(G, o)(v)$.

We denote by $\rho_{\mu}(s) \in B(\Hilb(\mu))$ the operator that assigns to $(G,o)$
the unitary operator on $\ell^2(\vertex(G))$ that sends $\delta_v$ to
$\delta_{v.i(s)}$.
It is easy to see that $\rho_{\mu}(s)$ is equivariant
for all $s \in S$, and that this assignment extends to a unitary
representation $\rho_{\mu} \colon \bF_S \to U(\Hilb(\mu))$ such that
$\rho_{\mu}(s)^* = \rho_\mu(i(s))$ for all $s \in S$.
This unitary
representation $\rho_{\mu} \colon \bF_S \to U(\Hilb(\mu))$ extends in turn
to a natural $*$-homomorphism $\rho_{\mu} \colon C(\sGG(S,X)) \rtimes \bF_S
\to B(\Hilb(\mu))$.
Indeed, consider the natural representation of $C(\sGG(S,X))$ on $\Hilb(\mu)$ by
multiplication, $f \mapsto M(f)$. In order to see that $M$ and $\rho_{\mu}
\colon \bF_S \to U(\Hilb(\mu))$ combine via linearity and $\rho_\mu(f w) :=
M(f) \rho_\mu(w)$ to a $*$-representation of the algebraic crossed product,
it suffices to check that
$$\forall w \in \bF_S, f \in C(\sGG(S,X)) \quad \rho_{\mu}(w) M(f)
\rho_{\mu}(w)^* = M({}^w\!f)\,,$$
as a simple verification using the definition of $C(\sGG(S,X)) \rtimes
\bF_S$ shows.
To prove that this equation indeed holds, we compute in $\ell^2(\verts(G))$
for $v \in \vertex(G)$ that
$$\rho_{\mu}(w) M(f) \rho_{\mu}(w)^*\delta_v = 
\rho_{\mu}(w) M(f) \delta_{v.w} = f(G,v.w) \cdot \rho_{\mu}(w) \delta_{v.w}
= M({}^{w}\!f)\delta_v\,.$$
This shows that $\rho_{\mu} \colon C(\sGG(S,X)) \rtimes \bF_S \to
B(\Hilb(\mu))$ exists as desired.

Now, if $\delta^{\mu} \in \Hilb(\mu)$ denotes the naturally defined
vector $(G,o) \mapsto \delta_o \in \ell^2(\verts(G))$, then 
$$\forall T \in C(\sGG(S,X)) \rtimes \bF_S \qquad \tau_\mu(T) = \langle
\rho_{\mu}(T) \delta^{\mu}, \delta^{\mu} \rangle\,.$$
Indeed, for all $w \in \bF_S, f \in C(\sGG(S,X))$, we have
\begin{eqnarray*}
\langle
\rho_{\mu}(fw) \delta^{\mu}, \delta^{\mu} \rangle
&=& \langle
M(f) \rho(w) \delta^{\mu}, \delta^{\mu} \rangle \\
&=& \int \langle
M(f) \rho(w) \delta_o, \delta_o \rangle \ d\mu(G,o) \\
&=& \int \langle
M(f) \delta_{o.w^{-1}}, \delta_o \rangle \ d\mu(G,o) \\
&=& \int \bfone_{o.w = o}  \cdot f(G,o)\ d \mu(G,o).
\end{eqnarray*}

This shows that $\tau_{\mu} \colon C(\sGG(S,X)) \rtimes \bF_S \to \bbC$ is
a positive linear functional. Although we shall not use it, we remark that
the Hilbert space $\Hilb(\mu)$ is the
GNS-construction associated with the trace $\tau_{\mu}$; see  \cite[Lemma 4 in Chapter 4]{Dixmier} for basics about the GNS-construction.
We denote by $R(\mu)$ the von Neumann algebra generated by the
$\rho_\mu$-image of
$C(\sGG(S,X)) \rtimes \bF_S$, i.e., $R(\mu):= \rho_{\mu}(C(\sGG(S,X))
\rtimes \bF_S)''$.
We call $R(\mu)$ the \dfnterm{von Neumann algebra of the unimodular random
network  $\mu$}. Since $F(C(\sGG(S,X)) \subset \rho_{\mu}(C(\sGG(S,X))
\rtimes \bF_S)'$, we conclude that $R(\mu) = \rho_{\mu}(C(\sGG(S,X))
\rtimes \bF_S)'' \subset F(C(\sGG(S,X))'$. Hence, every operator $T \in R(\mu)$ arises from a measurable family $(G,o) \mapsto T_{G,o}$.
We extend $\tau_{\mu}$ to a normal, positive linear functional
$\tr_{\mu}$ on $R(\mu)$ by the formula
\begin{equation} \label{deftrace}
\tr_{\mu}(T):= \langle T \delta^{\mu},\delta^{\mu} \rangle = \E\big[ \iprod{T_{\gh,o} \delta_o, \delta_o} \big] := \int \iprod{T_{\gh,o} \delta_o, \delta_o}
\,d\rtd(\gh, \bp)\,.
\end{equation}
The left-regular representation $\lambda_{\mu} \colon \bF_S \to U(\Hilb(\mu))$
is defined as acting on the underlying measure space, i.e., a vector $(G,o)
\mapsto \xi(G,o)$ is mapped via $\lambda_\mu(w)$
to $(G,o) \mapsto \xi(G,o.w)$. Using similar
arguments as above, we see that $F$ (rather than $M$) and $\lambda_{\mu}$
combine to give another representation $\lambda_{\mu} \colon C(\sGG(S,X))
\rtimes \bF_S \to  B(\Hilb(\mu))$ and we set $L(\mu):=
\lambda_{\mu}\big(C(\sGG(S,X)) \rtimes \bF_S\big)''$. It is now a matter of
checking definitions to see that an operator $T \in B(\Hilb(\mu))$ is
equivariant if and only if $T \in L(\mu)'$.

Put
$N_{\tau_{\mu}}:= \{T \in C(\sGG(S,X)) \rtimes \bF_S \st
\tau_{\mu}(T^*T)=0 \}$.
In order to put the players in the right framework, let us note that 
$\left(C(\sGG(S,X)) \rtimes \bF_S \right) /N_{\tau_{\mu}}$ together with
the inner-product $(T_1|T_2):= \tau_{\mu}(T_2^*T_1)$ is a Hilbert algebra in the
sense of \cite[Chapters 5 and 6]{Dixmier}. The algebras $L(\mu)$ and
$R(\mu)$ can be identified with the von Neumann algebras that are
left- and right-associated with this Hilbert algebra. Indeed, as we mentioned
above, the associated von Neumann algebras arise from the natural GNS-construction.

It follows from the Commutation Theorem \cite[Theorem 1 on page 80]{Dixmier} that $R(\mu) = L(\mu)'$, i.e., the operators in $R(\mu)$ are precisely the equivariant operators.
It was proved by \rref b.AL:urn/ that
$T \mapsto \int \iprod{T_{\gh,o} \delta_o, \delta_o}
\,d\rtd(\gh, \bp)$ is a trace on the algebra of equivariant operators. This
result is also an easy consequence of the general theory of Hilbert
algebras; see, for example, \cite[Theorem 1 on page 97]{Dixmier}.

Another useful point of view is to see $R(\mu)$ as the von Neumann algebra
associated with a discrete measured groupoid; see \cite{renault} for a
definition. Indeed, consider the $r$-discrete topological groupoid with
base space $\sGG(S,X)$ and an arrow between $(G,o)$ and $(G',o')$ for each
$v \in G$ such that $(G,v) \cong (G',o')$. Any unimodular measure $\mu$
turns this object into a discrete measured groupoid. The von Neumann
algebra $R(\mu)$ that has been described concretely above is the von
Neumann algebra associated with the discrete measured groupoid associated
to the measure $\mu$. We refer to \cite{feldmanmoore2} for details about the von Neumann algebra associated to a discrete measured equivalence relation and to \cite[Section 3]{sauer} or \cite{renault} for an extension to the realm of discrete measured groupoids.

Let us summarize:
\begin{thm} An operator $T \in B(\Hilb(\mu))$ is equivariant if and only if $T \in R(\mu)$.
The pair $(R(\mu),\tr_{\mu})$ is a tracial von Neumann algebra.
\end{thm}

We illustrate the definitions in two special cases.
(1) If $G=(\vertex,\edge)$ is a finite $S$-labelled Schreier network with automorphism
group $\Lambda$ and we consider the natural action of $\Lambda$ on $\ell^2
\vertex$, then there exists a natural isomorphism $$R(U(G)) \stackrel{\sim}{\to}
B(\ell^2 \vertex)^{\Lambda}:=\{ T \in B(\ell^2 \vertex) \st \forall \lambda
\in \Lambda \ T \lambda = \lambda T \}\,.$$
(2) If $\mu$ is concentrated on a Cayley diagram of a group $\Gamma$ (with
finite generating set $S$) and $X$ is a singleton, then $R(\mu) = R(\Gamma)$.

In complete analogy to the group case, we can define the von Neumann
algebra $R(\mu,S) \subset B(\Hilb(\mu,S))$, where $$\Hilb(\mu,S) :=
\int_{\sGG(S,X)} \ell^2(\edge(G)) \ d\mu(G,o)\,.$$
Again, there is a natural isomorphism $R(\mu,S) = M_S \left(R(\mu)
\right)$. We denote the natural trace on $R(\mu,S)$ by $\tr_{\mu} \colon R(\mu,S) \to \bbC$.

We shall now state and prove an embedding theorem for sofic unimodular
networks. The techniques are inspired by work of Elek and Lippner
\cite{eleklip} and P\u{a}unescu \cite{paunescu}. Again, the point is not to
give another proof of these results, but to prove new approximation formulas
that allow for applications to the approximation of the associated determinantal measures.

\begin{thm} \label{unimodapp}
Let $(G_n)_n$ be a sequence of finite $S$-labelled Schreier networks and let $\mu$
be a probability measure on $\sGG(S,X)$. Let $\omega$ be a non-principal
ultrafilter on $\bbN$. If 
$U(\gh_n)
\cd \rtd$, then there exists a trace-preserving embedding
$$\iota \colon (R(\mu),\tr_{\mu}) \to \prod_{n \to \omega} (R(U(G_n)),
\tr_{U(G_n)})\,.$$
Moreover, there exists a sequence of probability measures $(\nu_n)_n$ on
$\big(\{G_n\} \times \vertex(G_n)\big) \times \sGG(S,X)$ with the following properties:
\begin{enumerate}
\item[(i)] The measure $\nu_n$ has marginals $U(G_n)$ and $\mu$.
\item[(ii)] With respect to the natural metric $d_*$ on $\sGG(S,X)$, we have
$$\lim_{n \to \infty} \int d_*((G_n,v),(G,o)) \ d \nu_n((G_n, v),(G,o)) \to 0\,.$$
\item[(iii)] 
If $(T_n)_n \in \ell^{\infty}( \bbN, (R(U(G_n)), \tr_{U(G_n)}))$ represents $\iota(T)$ for some $T \in R(\mu)$, then
\begin{equation} \label{eq1c}
\lim_{n \to \omega} \int \left| 
\iprod{T_n \delta_{v.\gamma}, \delta_{v.\gamma'}} - \iprod{T_{G,o}
\delta_{o.\gpe}, \delta_{o.\gpe'}} \right| d \nu_n((G_n, v),(G,o)) =0 
\end{equation}
for all $\gpe, \gpe' \in \bF_S$.
\end{enumerate}
\end{thm}
\begin{proof} 
First of all, $\mu$ is unimodular. Recall the 
positive and unital trace $\tau_{\mu} \colon C(\sGG(S,X))
\rtimes \bF_S \to \bbC$ defined in \eqref{deftrace}. 
As in the proof of Proposition \ref{emb}, we consider
the unital $*$-homomorphism
$$\rho:= \lim_{n \to \omega} \rho_{U(G_n)} \colon C(\sGG(S,X)) \rtimes
\bF_S \to \prod_{n \to \omega} (R(U(G_n)), \tr_{U(G_n)})\,.$$
Since $U(G_n) \cd \mu$, we have that $\tau_{\mu} = \tr_{\omega} \circ \rho$,
where $\tr_{\omega}$
denotes the trace on the ultraproduct, i.e., $\tr_\omega((T_n)_n) :=
\lim_{n \to \omega} \tr_{U(G_n)}(T_n)$ for all norm-bounded sequences
$(T_n)_n$ with $T_n \in R(U(G_n))$. As $\tr_{\mu}$ is faithful on $R(\mu)$,
$\rho$ factors through the image $\im (\rho_\mu)$ of $\rho_{\mu} \colon
C(\sGG(S,X)) \rtimes \bF_S \to R(\mu)$. That is, there is a unique 
bounded $*$-homomorphism
$$
\psi
\colon \im (\rho_\mu) \to \prod_{n \to \omega} (R(U(G_n)), \tr_{U(G_n)})
$$
such that $\rho = \psi \circ \rho_\mu$.
By definition
$\im (\rho_{\mu})$ is weakly dense in $R(\mu)$, whence
$\psi$ extends to a
trace-preserving $*$-homomorphism $\iota$ from $R(\mu)$ to the
ultraproduct von Neumann algebra.

Weak convergence of measures on $\big(\sGG(S,X), d_*\big)$ is equivalent
(see the last corollary in \cite{Strassen} or \cite[3.1.1]{Skorohod})
to convergence in the Wasserstein metric
$$d_{\rm W}(\mu',\mu) := \inf_{\nu} \int d_*((G',o'),(G,o))\
d\nu((G',o'),(G,o))\,,$$ where
the infimum is taken over all measures $\nu$ on $\sGG(S,X) \times \sGG(S,X)$
with marginals $\mu'$ and $\mu$. Hence, we obtain a sequence of measures
$\nu_n'$ on $\sGG(S,X) \times \sGG(S,X)$ with marginals $U(G_n)$ and
$\mu$ so that 
$$\lim_{n \to \infty} \int d_*((G',o'),(G,o)) \ d \nu'_n((G',o'),(G,o)) \to
0\,.$$
Since the natural map $\big(G_n \times \verts(G_n)\big) \times \sGG(S,X) \to \sGG(S,X) \times
\sGG(S,X)$ is finite-to-one,
we can lift $\nu'_n$ to a measure
$\nu_n$ on $(G_n \times \verts(G_n)\big)\times \sGG(S,X)$. This proves (i) and (ii). 

It remains
to prove claim (iii) and in doing so, we follow the strategy of the proof of
Proposition \ref{emb}. Indeed, using the arguments in the proof of
Proposition \ref{emb}, it is again easy to see that the truth of (iii) 
depends only on $T \in R(\mu)$ and not on the choice of an approximating
sequence $(T_n)_n$. If $T$ lies in the image of $C(\sGG(S,X)) \rtimes
\bF_S$ and $T' = \sum_{w} f_w w$ is some choice of a preimage of $T$ in the
crossed product algebra, then there is a canonical approximating sequence
$(T_n)_n$ that represents $\iota(T)$. Indeed, for each $n$, there is a
$*$-homomorphism $\rho_{U(G_n)} \colon C(\sGG(S,X))
\rtimes \bF_S \to R(U(G_n))$ and we set $T_n := \rho_{U(G_n)}(T')$ for each $n \in
\bbN$.
For each $w \in \bF_S$, the function $f_w$ in the presentation of $T'$ is
uniformly continuous and hence $(G,o) \mapsto \langle T \delta_{o.\gamma},
\delta_{o.\gamma'} \rangle$ is uniformly continuous as well for such $T$.
Thus, (ii) easily implies (iii) for every element in the image of
$C(\sGG(S,X)) \rtimes \bF_S$. As in the proof of Proposition \ref{emb}, a
diagonalization
argument shows that (iii) holds for all $T \in R(\mu)$. This finishes the
proof. 
\end{proof}

\bsection{Existence of sofic monotone couplings}{s.sofic-couple}

Let $\mu$ be a unimodular probability measure on rooted networks.
Given an equivariant positive contraction $Q$ on $\Hilb(\mu)$,
we obtain a determinantal probability measure
$\P^{Q_G}$ on $\{0, 1\}^{\vertex(G)}$ associated to $\mu$-a.e.\ rooted network
$(G, o)$.
Note that if $G$ already has marks, then we regard $\P^{Q_G}$ as producing
(at random) new marks, $\eta \in \{0, 1\}^{\verts(G)}$,
which we may take formally as second coordinates after the existing marks.
In other words, we define the probability measure $\mu^Q$
by the equation
$$
\mu^Q\big[B(o, r; G) \cong (A, v),\ \eta \restrict C \equiv 1\big]
=
\int_{[B(o, r; G) \cong (A, v)]} \det(Q_G \restrict C) \,d\mu(G, o)
$$
for every rooted network $(A, v)$ of radius $r$ and every measurable choice
of $C \subseteq B(o, r; G)$.
Using involution invariance, it is easy to check that $\mu^Q$ is
unimodular.

As shown in \rref p.addS/, we may assume that $\mu$ is carried by
$S$-labelled Schreier networks.
In this case, it suffices to take the measurable choice
of $C \subseteq B(o, r; G)$ in the definition of $\mu^Q$ to be of the form
$\{o.w_1, \ldots, o.w_n\}$ for some $w_1, \ldots, w_n \in \FS$.

As a special case, let $G$ be a finite network and $Q$ be a positive
contraction on $\ell^2\big(\vertex(G)\big)$. Then $U(G)^Q$ is the
determinantal probability measure reviewed in \rref s.def/, regarded as a
randomly rooted network.

Given a unimodular probability measure $\mu$ with mark space $\marks$ and
two unimodular probability measures $\mu_i$ with mark spaces $\marks \times
\{0, 1\}$ for $i = 1, 2$, both of whose marginals forgetting the second
coordinate of the marks is $\mu$, we say that a unimodular probability measure
$\nu$ with the 3-coordinate 
mark space $\marks \times \{0, 1\} \times \{0, 1\}$ is a
\dfnterm{monotone coupling} of $\mu_1$ and $\mu_2$ if its marginal
forgetting the coordinate $4-i$ is $\mu_i$ for $i = 1, 2$ and $\nu$ is
concentrated on networks whose marks $(\xi, j, k)$
satisfy $j \le k$.

We shall prove the following extension of \rref t.gencouple/:

\procl t.gencoupleURN
Let $\mu$ be a sofic probability measure on rooted networks.
If $0 \le Q_1 \le Q_2 \le I$ in $R(\mu)$, then there
exists a sofic monotone coupling of $\P^{Q_1}$ and $\P^{Q_2}$.
\endprocl

As we noted, we may assume that $\mu$ is carried by
Schreier networks.
We have the following extension of
\rref l.ultradetl/:

\procl l.ultradetlURN
Let $(G_n)_n$ be a sequence of finite $S$-labelled Schreier
networks whose random weak limit is $\mu$. Let $\iota$ and $\iota_S$
be trace-preserving embeddings as in Theorem \ref{unimodapp}.
Let $T \in R(\mu)$ be such that $0 \leq T \leq I$ and suppose that $(T_n)_n$
represents $\iota(T)$ in the ultraproduct $\prod_{n \to \omega}
(B(\ell^2 \vertex_n),\tr_{\vertex_n})$ with $0 \leq T_n \leq I$ for
each $n \in \bbN$.
Then $\lim_{n \to \omega} U(G_n)^{T_n} = \mu^T$ in the weak topology.
\endprocl

The proof is essentially the same as that for \rref l.ultradetl/, using the
coupling probability measures $\nu_n$ of Theorem \ref{unimodapp} (but not
assuming any analogue of the injectivity of $\pi$).

We may now use \rref l.ultradetlURN/ to prove \rref t.gencoupleURN/,
just as \rref t.gencouple/ was proved.

All the above was for determinantal probability measures on subsets of
vertices. In order to deduce corresponding results for determinantal
probability measures on subsets of edges (or even ``mixed" measures on
subsets of both vertices and edges), we use the following construction.
Given a network $G$, subdivide each edge $e$ by adding a new vertex $x_e$,
which is joined to each endpoint of $e$ and which receives the mark of $e$.
Also assign a second coordinate to the new vertices so that we may
distinguish them.
Provided that the expected degree of the root under the unimodular measure
$\mu$ is finite, we may choose a re-rooting of the subdivided networks in
order to obtain a natural unimodular probability measure, $\tilde \mu$: see
\cite[Example 9.8]{AL:urn} for details.
In fact, when $\mu$ is the random weak limit of $(G_n)_n$,
we may simply subdivide the edges of $G_n$ and take the random weak limit
of the resulting networks, $G'_n$.
Using the finiteness of the expected degree under $\mu$,
it is not hard to check that $U(G'_n)$ does indeed converge to $\tilde \mu$.
With this construction, if we desire a determinantal probability measure on
the edges for $\mu$, we may simply use the corresponding positive
contraction on the vertices for $\tilde \mu$, where all entries are 0 that
do not correspond to a pair of new vertices.

In particular, this result allows to extend our observations on the existence of invariant monotone couplings between $\wsf$ and $\fsf$ to unimodular random networks. In combination with the results in \cite{abertviragthom}, we are also able to extend the results in Section \ref{s.approxim}. We omit proofs since the strategy and the techniques are unchanged.

\section*{Acknowledgments}

This research started when R.L. visited Universit\"at G\"ottingen as a guest of the Courant Research Centre G\"ottingen in November 2008 and was continued later when A.T.\ visited Indiana University at Bloomington as a Visiting Scholar in March 2011. A.T. thanks ERC for support.
We thank Ben Hayes for help with the proof of \rref l.dbarSchatten/.


\begin{thebibliography}{dd}

\bibitem{AW:Bernoulli}
{Ab{\'e}rt, M. and Weiss, B.},
Bernoulli actions are weakly contained in any free action.
{\em Ergodic Theory Dynam. Systems}, {\bf 33}(2) (2013), 323--333.

\bibitem{abertviragthom}
Ab{\'e}rt, M., Vir\'ag, B. and Thom, A.,
Benjamini-Schramm convergence and pointwise convergence of the spectral measure.
Preprint, (2011).

\bibitem{Adams}
{Adams, S.} ,
Very weak {B}ernoulli for amenable groups.
{\em Israel J. Math.}, {\bf 78}(2--3), (1992) 145--176.

\bibitem{AL:urn}
Aldous, D.J. and Lyons, R.,
Processes on unimodular random networks.
{\em Electron. J. Probab.} {\bf 12} (2007), no. 54, 1454--1508 (electronic).

\bibitem{bekka}
Bekka, M., Valette, A.,
Group cohomology, harmonic functions and the first $L^2$-Betti number. 
{\em Potential Anal.} {\bf 6}(4), (1997) 313--326.

\bibitem{BLPS:usf} 
Benjamini, I., Lyons, R., Peres, Y., and Schramm, O.,
Uniform spanning forests.
{\em Ann. Probab.} {\bf 29},  (2001) 1--65.

\bibitem{BS:rdl}
Benjamini, I. and Schramm, O.,
Recurrence of distributional limits of finite planar graphs.
{\em Electron. J. Probab.} {\bf 6} (2001b), no. 23, 13 pp. (electronic).

\bibitem{BBL:Rayleigh}
Borcea, J., Br{\"a}nd{\'e}n, P., and Liggett, T.M.,
Negative dependence and the geometry of polynomials.
{\em J. Amer. Math. Soc.} {\bf 22} (2009), 521--567.

\bibitem{bowen}
Bowen, L.,
Couplings of uniform spanning forests.
{\em Proc. Amer. Math. Soc.} {\bf 132} (2004), no. 7, 2151--2158.

\bibitem{Bowen:sofic} Bowen, L., {Measure conjugacy invariants for actions
of countable sofic groups}.  {\em J. Amer. Math. Soc.} {\bf 23} (2010), no.
1, {217--245}.

\bibitem{BSST}
Brooks, R.L., Smith, C.A.B., Stone, A.H., and Tutte, W.T.,
The dissection of rectangles into squares.
{\em Duke Math. J.} {\bf 7} (1940), 312--340.

\bibitem{BroOz}
Brown, N. and Ozawa, N.,
{\em $C^*$-Algebras and Finite-Dimensional Approximations.}
Cambridge U. Press, 1994.

\bibitem{BurPem}
Burton, R.M. and Pemantle, R.,
Local characteristics, entropy and limit theorems for spanning trees and domino tilings via transfer-impedances.
{\em Ann. Probab.} {\bf 21} (1993), 1329--1371.

\bibitem{ChifanIoana}
{Chifan, I. and Ioana, A.},
Ergodic subequivalence relations induced by a {B}ernoulli action.
{\em Geom. Funct. Anal.}, {\bf 20}(1) (2010), 53--67.

\bibitem{connes}
Connes, A.,
Classification of injective factors cases $II_1$, $II_{\infty}$, $III_\lambda$, $\lambda \neq 1$.
{\em  Ann. of Math.}, 2nd Series, {\bf 104}, No. 1 (1976), 73--115.

\bibitem{cornulier}
Cornulier, Y.,
A sofic group away from amenable groups.
{\em Math. Ann.} {\bf 350}:2 (2011), 269--275.

\bibitem{Dixmier}
Dixmier, J.,
{\em Von Neumann Algebras},
{North-Holland Mathematical Library} {\bf 27},
Amsterdam (1981), xxxviii+437.

\bibitem{Douglas}
Douglas, R.G.,
On majorization, factorization, and range inclusion of operators on Hilbert
space.
{\em Proc. Amer. Math. Soc.}
{\bf 17}, No. 2 (1966), 413--415.

\bibitem{ES:direct}
Elek, G. and Szab\'o, E.,
Sofic groups and direct finiteness.
{\em J. Algebra} {\bf 280} (2004), no. 2, 426--434. 


\bibitem{elekszabo}
Elek, G. and Szab\'o, E.,
Hyperlinearity, essentially free actions and $L^2$-invariants. The sofic property. 
{\em Math. Ann.} {\bf 332}:2 (2005), 421--441.

\bibitem{elekszabo2}
Elek, G. and Szab\'o, E., 
On sofic groups. 
{\em J. of Group Theory} {\bf 9}:2 (2006), 161--171.

\bibitem{eleklip}
Elek, G. and Lippner, G.,
Sofic equivalence relations, 
{\em J. Funct. Anal.} {\bf 258}:5 (2010), 1692--1708.

\bibitem{feldmanmoore2}
Feldman, J. and Moore, C.,
Ergodic equivalence relations, cohomology, and von Neumann algebras. II
{\em Trans. Amer. Math. Soc.} {\bf 234} (2) (1977) 325--359.

\bibitem{FontesMathieu}
Fontes, L.R.G. and Mathieu, P.,
On symmetric random walks with random conductances on {${\mathbb Z}\sp d$}.
{\em Probab. Theory Related Fields} {\bf 134} (2006), 565--602.

\bibitem{Glasner} Glasner, E., {\em Ergodic Theory via Joinings}.  {American
Mathematical Society}, {Providence, RI}.  (2003).

\bibitem{Hag:rcust}
H{\"a}ggstr{\"o}m, O.,
Random-cluster measures and uniform spanning trees.
{\em Stochastic Process. Appl.} {\bf 59} (1995), 267--275.

\bibitem{Houdayer}
{Houdayer, C.},
Invariant percolation and measured theory of nonamenable groups
  [after {G}aboriau-{L}yons, {I}oana, {E}pstein].
{\em Ast\'erisque}, {\bf 348} (2012), Exp. No. 1039, ix, 339--374.
S{\'e}minaire Bourbaki: Vol. 2010/2011. Expos{\'e}s 1027--1042.

\bibitem{KR1}
Kadison, R.V. and Ringrose, J.R.,
{\em Fundamentals of the Theory of Operator Algebras. {V}ol. {I}},
  volume 15 of {\em Graduate Studies in Mathematics}.
American Mathematical Society, Providence, RI. (1997).
Elementary theory, Reprint of the 1983 original.

\bibitem{KR}
Kadison, R.V. and Ringrose, J.R.,
{\em Fundamentals of the Theory of Operator Algebras. {V}ol. {II}},
  volume 16 of {\em Graduate Studies in Mathematics}.
American Mathematical Society, Providence, RI. (1997).
Advanced theory, Corrected reprint of the 1986 original.

\bibitem{KulTas}
Kulesza, A. and Taskar, B.,
Determinantal point processes for machine learning,
{\em Foundations and Trends in Machine Learning}, {\bf 5}, 2--3 (2012),
123--286.
DOI: 10.1561/2200000044.

\bibitem{kechrisatal}
Kechris, A., Solecki, S. and Todor\v{c}evi\'c, S., 
Borel chromatic numbers, 
{\em Advances in Mathematics} {\bf 141} (1999), 1--44.


\bibitem{Kirchhoff}
Kirchhoff, G.,
Ueber die {A}ufl\"osung der {G}leichungen, auf welche man bei der
  {U}ntersuchung der linearen {V}ertheilung galvanischer {S}tr\"ome gef\"uhrt
  wird.
{\em Ann. Phys. und Chem.} {\bf 72} (1847), 497--508.

\bibitem{Kun:Lip}
{Kun, G.},
Expanders have a spanning {L}ipschitz subgraph with large girth.
Preprint, \arXiv{1303.4982}, (2013).

\bibitem{lueckalt}
L\"uck, W.,
Approximating $L^2$-invariants by their finite-dimensional analogues. 
{\em Geom. Funct. Analysis} {\bf 4} (1994) 455--481. 

\bibitem{lueck}
L\"uck, W.,
 {\it $L^2$-Invariants: Theory and Applications to Geometry and K-Theory.} Ergebnisse der Mathematik und ihrer Grenzgebiete, {\bf 3}. A Series of Modern Surveys in Mathematics, 44. Springer, Berlin, 2002.


\bibitem{Lyons:bird}
Lyons, R.,
A bird's-eye view of uniform spanning trees and forests.
In Aldous, D. and Propp, J., editors, {\em Microsurveys in
  Discrete Probability}, volume 41 of {\em DIMACS Series in Discrete
  Mathematics and Theoretical Computer Science}, pages 135--162. Amer. Math.
  Soc., Providence, RI. (1998).
Papers from the workshop held as part of the Dimacs Special Year on
  Discrete Probability in Princeton, NJ, June 2--6, 1997.


\bibitem{Lyons:det}
Lyons, R.,
Determinantal probability measures.
{\em Publ. Math. Inst. Hautes \'Etudes Sci.} {\bf 98} (2003), 167--212.
Errata, \url{http://mypage.iu.edu/~rdlyons/errata/bases.pdf}.

\bibitem{LS:dyn}
{Lyons, R. and Steif, J.E.},
Stationary determinantal processes: Phase multiplicity,
  {B}ernoullicity, entropy, and domination.
{\em Duke Math. J.}, {\bf 120}(3) (2003), 515--575.

\bibitem{Lyons:fixed}
{Lyons, R.},
Fixed price of groups and percolation.
{\em Ergodic Theory Dynam. Systems}, {\bf 33}(1) (2013), 183--185.

\bibitem{Lyons:fiid}
{Lyons, R.},
Factors of IID on trees.
Preprint, \arXiv{1401.4197}, (2013).

\bibitem{Macchi}
Macchi, O., 
The coincidence approach to stochastic point processes,
{\em Advances in Appl. Probability} {\bf 7} (1975), 83--122.

\bibitem{Mester:mono}
Mester, P.,
Invariant monotone coupling need not exist.
{\em Ann. Probab.} {\bf 41}, 3A (2013), 1180--1190.

\bibitem{Nielsen}
Nielsen, O.A.,
{\em Direct Integral Theory}, volume 61 of {\em Lecture Notes in Pure
  and Applied Mathematics}.
Marcel Dekker Inc., New York. (1980).

\bibitem{Orn:factor}
{Ornstein, D.} ,
Factors of {B}ernoulli shifts are {B}ernoulli shifts.
{\em Advances in Math.}, {\bf 5}, (1970) 349--364.

\bibitem{Orn:book}
{Ornstein, D.S.},
{\em Ergodic Theory, Randomness, and Dynamical Systems}.
Yale University Press, New Haven, Conn., 1974.
James K. Whittemore Lectures in Mathematics given at Yale University,
  Yale Mathematical Monographs, No. 5.

\bibitem{OrnW:amen}
{Ornstein, D.S. and Weiss, B.},
Entropy and isomorphism theorems for actions of amenable groups.
{\em J. Analyse Math.}, {\bf 48}, (1987) 1--141.

\bibitem{paunescu}
P\u{a}unescu, L., 
On sofic actions and equivalence relations
{\em J. Funct. Anal.} {\bf 261}:9 (2011), 2461--2485.

\bibitem{Pemantle:ust}
Pemantle, R.,
Choosing a spanning tree for the integer lattice uniformly.
{\em Ann. Probab.} {\bf 19} (1991), 1559--1574.


\bibitem{petthom}
Peterson, J. and Thom, A., 
Group cocycles and the ring of affiliated operators.
{\em Invent. Math.} {\bf 185} (2011),  561--592.

\bibitem{Popa}
{Popa, S.},
Some computations of 1-cohomology groups and construction of
  non-orbit-equivalent actions.
{\em J. Inst. Math. Jussieu}, {\bf 5}(2), (2006) 309--332.


\bibitem{renault}
Renault, J.N., 
{\em A Groupoid Approach to $C^*$-algebras}, 
Lecture Notes in Math., no. 793, Springer-Verlag, 1980.

\bibitem{sauer}
Sauer, R.,
$L^2$-Betti numbers of discrete measured groupoids,
{\em Int. J. Algebra and Computation} {\bf 15}, 5 \& 6 (2005), 1169--1188.

\bibitem{Skorohod}
Skorohod, A.V.,
Limit theorems for stochastic processes. 
{\em Teor. Veroyatnost. i Primenen.} {\bf 1} (1956), 289--319. 
In Russian. English translation in {\em Theory Probab. Appl.} {\bf 1}
(1956), 261--290.

\bibitem{Soshnikov:survey} 
Soshnikov, A.,
Determinantal random point fields,
{\em Uspekhi Mat. Nauk} {\bf 55} (2000), 107--160.


\bibitem{Strassen}
Strassen, V.,
The existence of probability measures with given marginals. {\em Ann. Math. Statist.} {\bf 36} (1965), 423--439.

\bibitem{takesaki}
Takesaki, M.,
 {\it Theory of Operator Algebras. III.} Encyclopaedia of
Mathematical Sciences, {\bf 127}. Operator Algebras and Non-Commutative
Geometry, 8. Springer, Berlin, 2003.

\bibitem{thom}
Thom, A., 
Sofic groups and Diophantine approximation. 
{\em Comm. Pure Appl.
Math.}, {\bf 61}(8) (2008), 1155--1171.

\bibitem{Timar}
Tim{\'a}r, {\'A}.,
Ends in free minimal spanning forests.
{\em Ann. Probab.} {\bf 34} (2006), 865--869.

\bibitem{Weiss}
Weiss, B.,
Sofic groups and dynamical systems. 
Ergodic theory and harmonic analysis (Mumbai, 1999). {\em Sankhya Ser. A}
{\bf 62} (2000), no. 3, 350--359.

\end{thebibliography}
\end{document}